\documentclass{article}
\usepackage{amsrefs}
\usepackage{amssymb}
\usepackage{amsmath}
\usepackage{amsthm}
\usepackage{graphicx}

\newtheorem{theorem}{Theorem}
\newtheorem{proposition}{Proposition}
\newtheorem{lemma}{Lemma}
\newtheorem{definition}{Definition}
\newtheorem{corollary}{Corollary}
\newtheorem*{theoremA}{Theorem A}
\newtheorem*{theoremB}{Theorem B}
\DeclareMathOperator{\Pic}{Pic}

\DeclareMathOperator{\id}{id}
\DeclareMathOperator{\rk}{rank}
\DeclareMathOperator{\codim}{codim}
\DeclareMathOperator{\im}{im}
\DeclareMathOperator{\coker}{coker}
\DeclareMathOperator{\Tor}{Tor}
\DeclareMathOperator{\Spec}{Spec}
\DeclareMathOperator{\Ext}{Ext}
\DeclareMathOperator{\Hom}{Hom}
\newtheorem*{subject}{2000 Mathematics Subject Classification}
\newtheorem*{keywords}{Keywords}

\theoremstyle{remark}
\newtheorem{notation}{Notation}
\newtheorem{remark}{Remark}

\author{Marc Coppens\footnote{KU  Leuven, Department of Mathematics, Section of Algebra,
Celestijnenlaan 200B bus 2400 B-3001 Leuven and Technologiecampus Geel, department Elektrotechniek (ESAT),Kleinhoefstraat 4, B-2440 Geel Belgium; email: marc.coppens@kuleuven.be.}}
\title{A picture of the irreducible components of $W^r_d(C)$ for a general $k$-gonal curve $C$ }
\date{}

\begin{document}
\maketitle \noindent

\begin{abstract}
Based on results on Hurwitz-Brill-Noether theory obtained by H. Larson we give a picture of the irreducible components of $W^r_d(C)$ for a general $k$-gonal curve of genus $g$.
This picture starts from irreducible components of $W^r_d(C)$ restricted to an open subset of $\Pic (C)$ satisfying Brill-Noether theory as in the case of a general curve of genus $g$.
We obtain some degeneracy loci associated to a morphism of locally-free sheaves on them of the expected dimension.
All the irreducible components of the schemes $W^r_d(C)$ are translates of their closures in $\Pic (C)$.
We complete the proof that the schemes $W^r_d(C)$ are generically smooth in case $C$ is a general $k$-gonal curve (claimed but not completely proved before).
We obtain some results on the tangent spaces to the splitting degeneracy loci for an arbitrary $k$-gonal curve and we obtain some new smoothness results in case $C$ is a general $k$-gonal curve.
\end{abstract}

\begin{subject}
14H51
\end{subject}

\begin{keywords}
Hurwitz-Brill-Noether theory, splitting degeneracy loci, special divisors, Brill-Noether varieties, gonality, smoothness
\end{keywords}

\section{Introduction}\label{section1}

Let $C$ be a smooth curve of genus $g$ defined over an algebraically closed field $K$ and let $f:C \rightarrow \mathbb{P}^1$ be a morphism on degree $k$.
Let $M$ be the invertible sheaf on $C$ associated to $f$.

In \cite{ref7} Hannah Larson introduces so-called splitting loci $\Sigma_{\overrightarrow {e}}(C,f) \subset \Pic^d(C)$ (see Definition \ref{defs.l}) which gives a refinement of the stratification of $\Pic^d(C)$ by the sets $W^r_d(C) \setminus W^{r+1}_d(C)$ using the morphism $f$.
She also introduces associated splitting degeneracy loci $\overline{\Sigma}_{\overrightarrow{e}}(C,f)$ (see Definition \ref{defs.d.l}) containing $\Sigma_{\overrightarrow{e}}(C,f)$ as an open subset and being a closed subset of $\Pic (C)$.

In case $f$ is general she obtains results on so-called Hurwitz-Brill-Noether theory (see Theorem \ref{theorem1}).
Associated to $\overrightarrow {e}$ there is a number $u(\overrightarrow {e})$ (see Formula \ref{vgl.u}) and it is proved $\Sigma _{\overrightarrow {e}}(C,f)$ is not empty if and only if $u(\overrightarrow {e}) \leq g$ and in case $u(\overrightarrow {e})\leq g$ then $\dim (\Sigma _{\overrightarrow {e}}(C,f))=g-u(\overrightarrow {e})$.
It implies that in case $f$ is general one has $\overline {\Sigma}_{\overrightarrow {e}}(C,f)$ is the closure of $\Sigma_{\overrightarrow {e}}(C,f)$.
Also in case $u(\overrightarrow {e}) \leq g$ it is proved the set $\Sigma _{\overrightarrow {e}}(C,f)$ is smooth.
In particular, in case $C$ is a general $k$-gonal curve, one gets a description of the irreducible components of $W^r_d(C)$ as being certain splitting degeneracy loci $\overline{\Sigma}_{\overrightarrow {e}}(C,f)$ with $\overrightarrow {e}$ being of balanced plus balanced type (see Definition \ref{definition2} and Theorem \ref{theorem3}).

In this paper we introduce the open subspace $\Pic^d_f(C)$ of $\Pic^d(C)$ of line bundles free from $M$ (see Definition \ref{definition4}).
We consider the intersections $W^r_{d,f}(C)=W^r_d(C) \cap \Pic^d_f(C)$ and the results of Larson imply, in case $C$ is a general $k$-gonal curve, they satisfy the Brill-Noether Theory of a general smooth curve of genus $g$ in case $r \leq k-2$.
We show that all irreducible components of $W^r_d(C)$ are obtained from $W^r_{d,f}(C)$ in a natural way in case $C$ is a general $k$-gonal curve.

\begin{theoremA}
Let $C$ be a general $k$-gonal curve.
Then $W^r_{d,f}(C)$ is not empty if and only if $r \leq k-2$ and $\rho^r_d(g) \geq 0$ and if those conditions hold then $W^r_{d,f}(C)$ has dimension $\rho^r_d(g)$.
In such cases, on $W^r_{d,f}(C) \setminus W^{r+1}_{d,f}(C)$ there exists a morphism $u : F_1 \rightarrow F_2$ of locally free sheaves such that the associated degeneracy loci have the expected codimension.
Those degeneracy loci do measure the dimensions $h^0(C,L+M)$ for $L \in W^r_{d,f}(C) \setminus W^{r+1}_{d,f}(C)$.
The closures in $\Pic (C)$ of $W^r_{d,f}(C)$, those degeneracy loci and the translations of them by multiples of $M$ contain all irreducible components of all subsets $W^s_e(C)$ of all $\Pic^e(C)$.
\end{theoremA}

All statements in this theorem are contained in \cite{ref7}, except the description with the morphism $u$ of locally free sheaves (see Proposition \ref{prop2}).
On the other hand, this description starting with $W^r_{d,f}(C)$ gives a more natural description and shows in a more direct way how the other irreducible components show up.
It is remarkable that all irreducible components are obtained from special behaviour by adding $M$ only once (see also Remark \ref{rem3}).
This theorem also makes it natural to split up the irreducible components of $W^r_d(C)$ in case $C$ is a general $k$-gonal curve into three types (see Definition \ref{definition3}).

In \cite{ref7} it is also claimed that for a general $k$-gonal curve $C$ the schemes $W^r_d(C)$ are generically reduced (i.e. they have no multiple components).
However, as explained in Remark \ref{rem5}, the proof of that statement is not complete.
In this paper we finish the proof of that statement.

\begin{theoremB}
Let $C$ be a general $k$-gonal curve.
The schemes $W^r_d(C)$ are generically reduced.
\end{theoremB}

In \cite{ref6} one defines a scheme structure on the splitting degeneracy loci $\overline{\Sigma}_{\overrightarrow{e}}(C,f)$.
As noticed to me by H. Larson the smoothness results from \cite{ref7}, together with earlier results from \cite{ref13}, imply $\Sigma_{\overrightarrow{e}}(C,f)$ is a smooth scheme if $f$ is general.
However, in case $\overline{\Sigma}_{\overrightarrow{e}}(C,f)$ is an irreducible component of $W^r_d(C)$ then from the definition of the scheme structure of $\overline{\Sigma}_{\overrightarrow{e}}(C,f)$ we only know that it is a closed subscheme of $W^r_d(C)$.
As a matter of fact, as a scheme $\overline{\Sigma}_{\overrightarrow{e}}(C,f)$ is an intersection of schemes defined by certain Fitting ideals, while $W^r_d(C)$ as a scheme is defined by only one of those Fitting ideals (see Definitions \ref{definition5} and \ref{definition6}).

For any $f$ we prove the tangent spaces to $\overline{\Sigma}_{\overrightarrow{e}}(C,f)$ at a point $L \in \Sigma_{\overrightarrow {e}}(C,f)$ is determined by a restricted number of Fitting ideals (see Proposition \ref{prop4}).
Using this we conclude that in case $C$ is a general $k$-gonal curve then irreducible components of $W^r_d(C)$ of type I are not multiple components (see Corollary \ref{cor1}).
In order to prove that irreducible components of $W^r_d(C)$ of type II are not multiple components, we first prove some results from \cite{ref13} mentioned above in a slightly different way (see Lemma \ref{lemma6} and Proposition \ref{prop5}).
Finally in order to prove that for a general $k$-gonal curve the irreducible components of type III of $W^r_d(C)$ are not multiple components (see Corollary \ref{cor2}), in general we prove some more equalities of tangent spaces to the schemes defined by the Fitting ideals (see Proposition \ref{prop6}) for arbitrary $f$.
Those Propositions \ref{prop4} and \ref{prop6} on the tangent spaces of schemes defined by the Fitting ideals might be useful in the study of linear systems on special types of $k$-gonal curves.

We already mentioned that from \cite{ref7} it follows the schemes $\overline{\Sigma}_{\overrightarrow {e}}(C,f)$ are smooth along $\Sigma_{\overrightarrow{e}}(C,f)$ in case $f$ is general.
As an application of the methods used in this paper we find some more splitting loci contained in the smooth locus of $\overline{\Sigma}_{\overrightarrow{e}}(C,f)$ in case this is a component of type I of some $W^r_d(C)$.
In a forthcoming paper I hope to be able to extend those results.

In Section \ref{section2} we recall the main definitions and results from \cite{ref7} (and a part of \cite{ref6}).
In Section \ref{section3} we prove Theorem A.
In Section \ref{section4} we prove Theorem B.
In Section \ref{section5} we apply the methods of this paper to give the new smoothness result on $\overline{\Sigma}_{\overrightarrow{e}}(C,f)$.

\subsection{Some notations}

Since $\Pic (C)$ is an abelian group using tensor products of line bundles we use the additive notation.
In particular we write $L_1 + L_2$ to denote the tensor product $L_1 \otimes _{\mathcal{O}_C} L_2$, $-L$ to indicate the dual line bundle $\Hom (\mathcal{O}_C,L)$, $nL$ with $n \in \mathbb{Z}$ to indicate the $n$-th tensor power of $L$ in case $n>0$ and of the $-n$-th tensor power of $-L$ in case $n<0$ (if the notation would occur for $n=0$ then it means the structure sheaf $\mathcal{O}_C$).
In case $D$ is a divisor on $C$ then we also use such notations replacing $\mathcal{O}_C(D)$ by $D$.
We use $K$ to denote the canonical sheaf on $C$.
We write $h^0(L)$ (resp. $h^1(L)$) to denote the dimension of $H^0(C,L)$ (resp. $H^1(C,L)$).
In case $W$ is a subset of $\Pic^d(C)$ then we write $W+nM \subset \Pic^{d+kn}(C)$ to denote the subset $\{L+nM : M \in W \}$.
In case $C=\mathbb{P}^1$ and $\mathcal{F}$ is a sheaf on $\mathbb{P}^1$ and $n \in \mathbb{Z}$, then we write $\mathcal{F} (n)$ to denote the sheaf $\mathcal{F} \otimes \mathcal{O}_{\mathbb{P}^1}(n)$.

In case $\mathcal{F}$ is a sheaf on a scheme $T$ of finite type over $K$ and $x$ is a closed point of $T$ then as usual we write $\mathcal{F}_x$ to denote the stalk of $\mathcal{F}$ at $x$.
We write $K(x) \cong K$ to denote the quotient $\mathcal{O}_{X,x} / \mathcal{M}_{X,x}$.
We call $\mathcal{F}_x \otimes _{\mathcal{O}_{X,x}} K(x)$ the fibre of $\mathcal{F}$ at $x$ and we denote it by $\mathcal{F}(x)$.
In case $u : \mathcal{F}_1 \rightarrow \mathcal{F}_2$ is a morphism of sheaves on $T$ then we write $u(x)$ to the denote the induced morphism $\mathcal{F}_1(x) \rightarrow \mathcal{F}_2(x)$.

\section{Definitions from Hurwitz-Brill-Noether theory}\label{section2}

Let $C$ be a smooth projective curve of genus $g$ defined over an algebraically closed field $K$.
Let $f : C \rightarrow \mathbb{P}^1$ be a fixed morphism of degree $k$.
Let $M$ be the line bundle of degree $k$ on $C$ corresponding to $f$.

Let $L$ be an invertible sheaf on $C$ of degree $d$.
From Grauert's Theorem (see e.g. \cite{ref1}, Chapter III, Corollary 12.9) it follows $E := f_*(L)$ is a vectorbundle of degree $k$ on $\mathbb{P}^1$.
From a theorem of Grothendieck (see e.g. \cite{ref2}, Chapter I, Theorem 2.1.1) it follows $E$ splits as a direct sum of line bundles on $\mathbb{P}^1$.
This means there exist a sequence of integers $e_1 \leq e_2 \leq \cdots \leq e_k$ such that $E \cong \mathcal{O}_{\mathbb{P}^1}(e_1) \oplus \mathcal{O}_{\mathbb{P}^1}(e_2) \oplus \cdots \oplus \mathcal{O}_{\mathbb{P}^1}(e_k)$.
We write $\overrightarrow{e}$ to denote such sequence and $\mathcal{O}(\overrightarrow{e})$ to denote that direct sum.
We call $\overrightarrow{e}$ the splitting type of $L$ with respect to $f$ (and often we just say the splitting type of $L$).
Such a sequence of integers without referring to some invertible sheaf is called a splitting type.

Applying the Grotendieck -Riemann-Roch Theorem (see e.g. \cite{ref3}, Theorem 15.2) to the morphism $f$ and the invertible sheaf $L$ on $C$ we obtain
\begin{equation}\label{vgl1}
\sum _{i=1}^k e_i = d-g+1-k \text { .}
\end{equation}
From the projection formula (see e.g. \cite{ref1}, Chapter 2, Exercise 5.1(d)) it follows that for each $n \in \mathbb{Z}$ we have
\[
f_*(L +nM)=\mathcal{O}(\overrightarrow {e}+n)
\]
where we write $\overrightarrow {e}+n$ to the denote the splitting type $(e_1+n, \cdots ,e_k+n)$.
This implies for each $n \in \mathbb{Z}$ we have
\begin{equation}\label{vgl2}
h^0(L +nM)=\sum_{i=1}^k \max \{ 0, e_i+n+1 \} \text { .}
\end{equation}

Associated to the curve $C$ we have the Brill-Noether schemes $W^r_d(C)$.
As sets (meaning with their reduced scheme structure) those are defined by
\[
W^r_d(C)=\{ L \in \Pic^d(C) : h^0(L)\geq r+1 \} \text { .}
\]
Given the morphism $f$ this is refined in \cite{ref7} using the so-called splitting loci.

\begin{definition}\label{defs.l}
Associated to a splitting type $\overrightarrow {e}$ we define the Brill-Noether splitting locus
\[
\Sigma _{\overrightarrow {e}}(C,f) =\{ L \in \Pic^d(C) : f_*(L) \cong \mathcal{O}(\overrightarrow {e}) \} \text{  .}
\]
\end{definition}

Using Equation \ref{vgl2} and the uppersemicontinuity of $h^0(L_t)$ (see Notation \ref{not1}) for a family of line bundles on $C$ (see e.g. \cite{ref1}, Chapter III, Theorem 12.8) one proves the following.
\begin{lemma}\label{lemma1}
Assume $L \in \Pic^d(C)$ with $L \in \overline{\Sigma _{\overrightarrow {e}}(C,f)}$ (the Zariski closure) then $f_*(L)\cong \mathcal{O}(\overrightarrow {e}')$ for some $\overrightarrow {e}' \leq \overrightarrow {e}$.
\end{lemma}
In this lemma we use the following definition.
\begin{definition}\label{definition1}
Let $\overrightarrow {e}$,$\overrightarrow {e}'$ be two splitting types satisfying Formula \ref{vgl1} then
\[
\overrightarrow {e}' \leq \overrightarrow {e} \Leftrightarrow 1 \leq l \leq k : \sum_{i=1}^l e'_i \leq \sum_{i=1}^l e_i \text { .}
\]
\end{definition}
Now Lemma \ref{lemma1} motivates the following definition from \cite{ref7}.

\begin{definition}\label{defs.d.l}
Associated to a splitting type $\overrightarrow{e}$ we define the splitting degeneracy locus
\[
\overline {\Sigma}_{\overrightarrow {e}}(C,f)= \bigcup_{\overrightarrow {e'} \leq \overrightarrow {e}} \Sigma _{\overrightarrow {e}'}(C,f) \text { .}
\]
\end{definition}

The splitting degeneracy loci are closed subsets of $\Pic^d(C)$.
In general those closed subsets can be larger than the closure $\overline {\Sigma _{\overrightarrow {e}}(C,f)}$.

\begin{notation}\label{not1}
Now let $T$ be a scheme of finite type over $K$ and let $\mathcal{L}$ be a family of invertible sheaves of degree $d$ on $C$ over $T$ (hence $\mathcal{L}$ is an invertible sheaf on $C\times T$).
For a closed point $t \in T$ we write $L_t$ to denote $\mathcal{L} \vert _{C \times \{ t\}}$ and $E_t = f_*(L_t)$.
Let $f_T = f \times \id_T : C \times T \rightarrow \mathbb{P}^1 \times T$.
Although $T$ need not be integral, from the fact that $f$ is a finite morphism, it still follows $E_T:=(f_T)_*(\mathcal{L})$ is a locally free sheaf of rank $k$ on $\mathbb{P}^1 \times T$ and for each closed point $(y,t) \in \mathbb{P}^1 \times T$ we have $E_T (y,t) \rightarrow \mathcal{L} \vert _{f_T^{-1}(y,t)}$ is an isomorphism (here $f_T^{-1}(y,t)$ is considered as a scheme).
In case $\mathcal{L}$ is equal to the Poincar\'e sheaf $\mathcal{P}_d$ on $C \times \Pic^d(C)$ then we write $E_d$ instead of the locally free sheaf $E_{\Pic^d(C)}$ on $\mathbb{P}^1 \times \Pic^d(C)$.
\end{notation}

\begin{lemma}\label{lemma2}
Using Notation \ref{not1}, for a closed point $t \in T$ we have $E_T \vert _{\mathbb{P}^1 \times \{ t \}} \cong E_t$.
In particular we obtain
\[
\Sigma _{\overrightarrow {e}}(C,f) = \{ t \in \Pic^d(C) : E_d\vert _{\mathbb{P}^1 \times \{ t \}} \cong \mathcal{O}(\overrightarrow {e}) \} \text { .}
\]
\end{lemma}
\begin{proof}
From e.g. \cite{ref1}, Chapter III, Remark 9.3.1, we obtain a natural map $u_t : E_T \vert _{\mathbb{P}^1 \times \{ t \}} \rightarrow E_t$.
For each $y \in \mathbb{P}^1$ the isomorphism $E_T (y,t) =(E_T \vert _{\mathbb{P}^1 \times \{ t \}})(y,t) \rightarrow \mathcal{L}\vert _{f_T^{-1}(y,t)}$ factors as a composition of $u_t(y,t)$ and the isomorphism $E_t(y,t) \rightarrow \mathcal{L}\vert _{f_T^{-1}(y,t)}$.
This shows $u_t(y,t)$ is an isomorphism for all $y \in \mathbb{P}^1$.
Since $\mathbb{P}^1 \times \{ t \}$ is integral and $E_T \vert _{\mathbb{P}^1 \times \{ t \}}$ and $E_t$ are locally free sheaves this implies $u_t$ is an isomophism (see further on Lemma \ref{lemma4}).
\end{proof}

For a splitting sequence $\overrightarrow {e}$ we define the number
\begin{equation}\label{vgl.u}
u(\overrightarrow {e})=\sum _{1 \leq i < j \leq k} \max \{0, e_j-e_i-1 \} \text { .}
\end{equation}

Using Lemma \ref{lemma2} then from the second part of Theorem 14.7 in \cite {ref4} it follows that, in case $\Sigma _{\overrightarrow {e}} (C,f) \neq \emptyset$,
\begin{equation}\label{vgl3}
\dim (\Sigma _{\overrightarrow {e} } (C,f)) \geq g-u(\overrightarrow{e}) \text { .}
\end{equation}
The proof of that theorem is only written down in case $k=2$, but for the most important cases in this paper we are going to find an alternative proof.
Also, using Formula \ref{vgl2} the degeneracy splitting loci can be described as follows.
\begin{equation}\label{vgl4}
\begin{split}
\overline{\Sigma}_{\overrightarrow {e}}(C,f) = \{ t \in \Pic^d(C) : h^0(E_t \otimes \mathcal{O}_{\mathbb{P}^1}(n)) \geq \sum _{i=1}^k \max \{0, e_i+n+1 \} \text { for all } n \in \mathbb{Z} \} \\
=\{ t \in \Pic^d(C) : h^1 (E_t \otimes \mathcal{O}_{\mathbb{P}^1}(n)) \geq \sum _{i=1}^k \max \{ 0, -e_i-n-1 \} \text { for all } n \in \mathbb{Z} \} \text { .}
\end{split}
\end{equation}

In so-called Hurwitz-Brill-Noether theory one studies those degeneracy splitting loci in case $f$ is general (so in case $k$ is small enough with respect to $g$, one considers general $k$-gonal curves).
In \cite {ref7} one obtains the following results.

\begin{theorem}\label{theorem1}
Let $f$ be a general morphism.
Then $\Sigma _{\overrightarrow{e}} (C,f)$ is non-empty if and only is $u(\overrightarrow {e})\leq g$ and if $u(\overrightarrow {e}) \leq g$ then $\Sigma _{\overrightarrow {e}}(C,f)$ is smooth of dimension $g-u(\overrightarrow {e})$.
\end{theorem}

By a general morphism we mean it corresponds to a general point of the Hurwitz space of degree $k$ and genus $g$.
Because of the knowledge of the irreducibility of that Hurwitz space, in \cite{ref7} (and also in \cite{ref6}) it is assumed that the characteristic of $K$ is equal to zero or it does not divide $k$.
Recently in \cite{refirrH} it is proved that this Hurwitz space is irreducible without any restriction on the characteristic of $K$.
Notice that up to now we only consider sets, but see also Proposition \ref{prop3}.
Using tropical geometry it is proved that, if $\Sigma _{\overrightarrow {e}}(C,f)$ is not empty then the dimension of it is at most $g-u(\overrightarrow {e})$ in \cite{ref12}.
It also follows from Theorem \ref{theorem1} that $\overline{\Sigma}_{\overrightarrow {e}}(C,f)=\overline{\Sigma _{\overrightarrow {e}}(C,f)}$ (i.e; the degeneracy splitting locus is the closure of the splitting locus in case $f$ is general).
In \cite{ref6} one also proves the following result.

\begin{theorem}\label{theorem2}
Let $f$ be a general morphism.
In case $u(\overrightarrow {e}) <g$ then $\overline{\Sigma}_{\overrightarrow {e}}(C,f)$ is irreducible.
\end{theorem}

\section{The irreducible components of $W^r_d(C)$ as sets}\label{section3}

From the previous theorems \ref{theorem1} and \ref{theorem2} one obtains in \cite{ref7} a description of the irreducible components of the spaces $W^r_d(C)$ in case $C$ is a general $k$-gonal curve.

\begin{definition}\label{definition2}
A splitting type $\overline{e}$ is called of balanced plus balanced type if it is of the form
\[
(\overbrace{-b-1, \cdots , -b-1} ^x, \overbrace{-b, \cdots, -b}^y, \overbrace {a, \cdots, a}^u, \overbrace {a+1, \cdots , a+1}^v)
\]
for some $b \geq 1$, $a \geq 0$, $y>0$, $u>0$ ($x$ and/or $v$ might be 0).
Shortly we say $\overrightarrow {e}$ is of type $B(a,b,y,u,v)$ (this is different from the notation used in \cite{ref7}).
\end{definition}

Notice that for $L \in \Sigma _{\overrightarrow {e}}(C,f)$ with $\overrightarrow {e}$ of type $B(a,b,y,u,v)$ it follows from Equation \ref{vgl2} that
\begin{equation}\label{vgl5}
h^0(L)=u(a+1)+v(a+2) \text { .}
\end{equation}

The following statement is another formulation for most part of Corollary 1.3 in \cite{ref7}.

\begin{theorem}\label{theorem3}
The irreducible components of $W^r_d(C)$ for a general $k$-gonal curve $C$ are the irreducible components of $\overline{\Sigma}_{\overrightarrow {e}}(C,f)$ with $\overrightarrow {e}$ satisfying equation \ref{vgl1} and being of type $B(a,b,y,u,v)$ with $u(a+1)+v(a+2)=r+1$ and $b \geq 2$ or $(b,v)=(1,0)$.
In case $u(\overrightarrow {e})=g$ then this a finite set of points, if $u(\overrightarrow {e})<g$ then this is irreducible.
\end{theorem}

Based on this theorem we distinguish between three types of irreducible components for the schemes $W^r_d(C)$ of a general $k$-gonal curve.

\begin{definition}\label{definition3}
Let $C$ be a general $k$-gonal curve and let $W$ be an irreducible of some $W^r_d(C)$ corresponding to a splitting sequence $\overrightarrow {e}$ of type $B(a,b,y,u,v)$.
\begin{enumerate}
\item $W$ is of type I if $v=0$ and $a=0$
\item $W$ is of type II if $v \neq 0$ and $a=0$
\item $W$ is of type III otherwise .
\end{enumerate}
\end{definition}

Suppose  $W=\overline{\Sigma}_{\overrightarrow {e}}(C,f)$ is of type I.
Then for $L \in \Sigma _{\overrightarrow {e}}(C,f)$ we have $r=u-1 \leq k-2$ and $h^0(L-M)=0$.
Motivated by this observation we introduce the following definition.

\begin{definition}\label{definition4}
We say $L \in \Pic^d(C)$ is free from $M$ in case $h^0(L-M)=0$.
\end {definition}

\begin{lemma}\label{lemma3}
If $L \in \Pic^d(C)$ with $h^1 (L) \neq 0$ is free from $M$ then $h^0(L) \leq k-1$.
\end{lemma}
\begin{proof}
Assume $h^0(L) \geq k$.
Because of the base point free pencil trick (see e.g. \cite{ref9}, p.126) we have $h^0(L+M) \geq 2h^0(L) \geq h^0(L)+\deg (M)$.
Then for each $P \in C$ we have $h^0(L+P)=h^0(L)+1$, this is only possible if $h^1(L)=0$.
\end{proof}

\begin{notation}\label{not2}
We write $\Pic^d_f(C)= \{ L \in \Pic^d(C) : L \text { is free from } M \}$ and we write $W^r_{d,f}(C)= \Pic^d_f(C) \cap W^r_d(C)$.
\end{notation}

\begin{proposition}\label{prop1}
Let $C$ be a general $k$-gonal curve and assume $r>d-g$ and $r \geq 1$.
The irreducible components of the schemes $W^r_d(C)$ of type I are the closures in $\Pic^d(C)$ of the irreducible components of $W^r_{d,f}(C)$.
\end{proposition}
\begin{proof}
Let $W \subset W^r_d(C)$ be an irreducible component of type I, so $W=\overline{ \Sigma}_{\overrightarrow {e}}(C,f)$ with $\overrightarrow {e}$ is of type $B(0,b,y,u,0)$.
Because of Formule \ref{vgl2}, if $L \in \Sigma_{\overrightarrow {e}}(C,f)$ then $h^0(L-M)=0$, hence $L \in W^r_{d,f}(C)$.
So $W \cap \Pic^d_f(C) \neq \emptyset$ and since $\Pic^d_f(C)$ is an open subset of $\Pic^d(C)$ it follows $W \cap \Pic^d_f$ is an irreducible subset of $W^r_{d,f}(C)$ and its closure in $\Pic^d(C)$ is equal to $W$.
Let $Z$ be the irreducible component of $W^r_{d,f}(C)$ containing $W \cap \Pic^d_f$. 
Its closure $\overline{Z}$ in $\Pic^d(C)$ is irreducible and satisfies $W \subset \overline{Z} \subset W^r_d(C)$.
It follows $W=\overline {Z}$ hence $Z=W \cap \Pic^d_f$.
This proves $W$ is the closure in $\Pic^d(C)$ of an irreducible component of $W^r_{d,f}$.

Let $Z$ be an irreducible component of $W^r_{d,f}(C)$, let $\overline {Z} \subset W^r_d(C)$ be its closure in $\Pic^d(C)$ and let $W$ be an irreducible component of $W^r_d(C)$ containing $\overline{Z}$.
For an element $L'$ of $Z$ we have $h^0(L'-M)=0$, hence for a general element $L$ of $W$ we also have $h^0(L-M)=0$.
This implies $W$ is an irreducible component of type I.
It follows $W \cap \Pic^d_f(C)$ is an irreducible subset of $W^r_{d,f}(C)$ containing $Z$, hence equal to $Z$.
Since the closure of $W \cap \Pic^d_f(C)$ in $\Pic^d (C)$ is equal to $W$ it follows $W=\overline{Z}$.
This proves $Z$ is the intersection of a component of type I of $W^r_d(C)$ with $\Pic^d_f(C)$.
\end{proof}

From Theorem \ref{theorem1} it follows that is $W \subset W^r_d(C)$ is an irreducible component of type I (and $C$ a general $k$-gonal curve) then $\dim (W)=\rho ^r_d(g)$.
From Theorem \ref{theorem2} it follows this component is unique (given $d$ and $r$) in case $\rho ^r_d(g)>0$.
Combining those remarks with Proposition \ref{prop1} we can conclude that the varieties $W^r_{d,f}(C)$ on a general $k$-gonal curve satisfy the classical Brill-Noether dimension and irreducibility statements in case $C$ is a general $k$-gonal curve.
The existence of components of type I satisfying the dimension statement was already proved in \cite{ref11}.
As it will be explained further on, those components of type I can be considered as the basic components for all other components taking into account the existence of $M$.

As a first remark concerning that claim, we consider irreducible components of type III.
Let $W=\overline{\Sigma}_{\overrightarrow {e}}(C,f)$ be a component of type III.
Then $\overrightarrow {e}=\overrightarrow {e'}+a$ with $\overrightarrow {e'}$ of type $B(0,b-a,y,u,v)$.
So $\overline{\Sigma}_{\overrightarrow {e'}}(C,f)$ is an irreducible component $W'$ of $W^{r-a(u+v)}_{d-ak}(C)$ of type I or II and we have $W=W'+aM$.
Therefore the irreducible components of type III can be considered as translations of the irreducible components of types I or II.

Up to now, this is a way to summarize the results on the irreducible components of $W^r_d(C)$ for a general $k$-gonal curve $C$ obtained in \cite{ref7}.
The following proposition gives a natural description of the irreducible components of type II not contained in \cite{ref7}.

\begin{proposition}\label{prop2}
Let $C$ be a general $k$-gonal curve and let $W$ be a component of $W^r_d(C)$ of type I.
On $U=(W \cap \Pic^d_f(C))\setminus W^{r+1}_d(C)$ there exists a map of locally free sheaves $u : F_1 \rightarrow F_2$ such that the degeneracy loci of $u$ have the expected codimension (and are empty if this expected codimension is larger than $\dim (W)$).
Let $U'$ be an irreducible component of such degeneracy locus different from $U$ and let $W'$ be its closure in $\Pic^d(C)$.
Then $W'+M \subset \Pic^{d+k}(C)$ is equal to a splitting degeneracy locus $\overline{\Sigma}_{\overrightarrow {e'}}(C)$ with $\overrightarrow {e'}$ of some type $B(0,b',y',u',r+1)$.
In particular in case $b' \geq 2$ (notation from Definition \ref{definition2}) then this is a component of type II.
All irreducible components of type II can be obtained in this way.
\end{proposition}

\begin{proof}
Let $W=\overline{\Sigma}_{\overrightarrow{e}}(C,f)$ with $\overrightarrow {e}$ of type $B(0,b,y,r+1,0)$.
Let $\overrightarrow {e'}$ be of type $B(0,b',y',t,r+1)$ for some $t \geq 1$ and such that $\overrightarrow {e'}-1 \leq \overrightarrow {e}$ then of course $\Sigma _{\overrightarrow {e'}-1}(C) \subset U$.
Clearly each splitting sequence of balanced plus balanced type with $a=0$ and $v\neq 0$ can be obtained in that way (using a suited $\overrightarrow {e}$).

Let $\mathcal{L}$ be the restriction of the Poincar\'e bundle $\mathcal{P}_d$ to $U \times C$.
Let $p: U\times C \rightarrow C$ and $\pi : U \times C \rightarrow U$ be the projection morphisms.
Let $D$ be an effective divisor of large degree on $C$ and consider the natural exact sequence
\[
0 \rightarrow \mathcal{L} \otimes p^*(M) \rightarrow \mathcal {L} \otimes p^*(M+D) \rightarrow \mathcal{L} \otimes p^*(M+D)\vert _{U \times D} \rightarrow 0 \text { .}
\]
Using this sequence we obtain the following exact sequence on $U$
\[
\pi _*(\mathcal{L} \otimes p^*(M+D)) \rightarrow \pi _*(\mathcal {L} \otimes p^*(M+D)\vert _{U \times D}) \rightarrow R^1\pi_* (\mathcal {L} \otimes p^* (M)) \rightarrow 0 \text { .}
\]
In this sequence $\pi _*(\mathcal {L} \otimes p^*(M+D))$ is locally free of rank $d+k+\deg (D)-g+1$ and $\pi_* (\mathcal {L} \otimes p^*(M+D) \vert _{U\times D})$ is locally free of rank $\deg (D)$.
Remember that for each $L \in U$ we have $h^1(L+M+D)=0$ (implying $R^1\pi_* (\mathcal {L}\otimes p^*(M+D))=0$) and since $R^2\pi_*(\mathcal {L}\otimes p^*(M))=0$ we have $R^1\pi_*(\mathcal {L} \otimes p^*(M))(L) \cong H^1(C,L+M)$ for each $L \in U$.
Those remarks follow from Grauert's Theorem and the cohomology and base change Theorem , see e.g. \cite {ref1}, Chapter III, Section 12.

Let $s$ and $t$ be two general global sections of $M$ and consider the map $\mathcal{L} \oplus \mathcal {L} \rightarrow \mathcal {L} \otimes p^*(M)$ locally defined by $(u,v) \rightarrow u \otimes s+v\otimes t$.
Consider the resulting composite morphism between locally free sheaves on $U$:
\[
q : \pi _*(\mathcal {L} \oplus \mathcal {L}) \rightarrow \pi_*(\mathcal {L} \otimes p^*(M)) \rightarrow \pi_*(\mathcal {L} \otimes p^*(M+D))
\]
For $L\in U$ the map
\[
q(L) : H^0(C,L)\oplus H^0(C,L) \cong H^0(C,L)\otimes H^0(C,M) \rightarrow H^0(C,L + M + D)
\]
factors through the natural multiplication map
\[
H^0 (C,L) \otimes H^0 (C,M) \rightarrow H^0(C,L+M) \text { .}
\]
with $H^0 (C,L+M) \rightarrow H^0(C,L+M+D)$ being injective.

From the base point free pencil trick (see e.g. \cite {ref9}, p. 126) if follows its kernel is $H^0 (C, L-M)$.
Since $[L] \in \Pic^d_f(C)$ this is zero, hence $q(L)$ is injective for every $L \in U$.
This implies $q$ is a monomorphism and the cokernel $N$ is also locally free (see Lemma \ref{lemma4}) of rank $d+k+\deg (D)-g+1-2(r+1)$.
By construction the composition of $q$ with $\pi_*(\mathcal{L} \otimes p^*(M+D)) \rightarrow \pi_*(\mathcal {L} \otimes p^* (M+D)\vert _{U \times D})$ is zero, so we obtain an exact sequence
\[
N \xrightarrow{u} \pi_*(\mathcal{L} \otimes p^*(M+D)\vert_{U \times D}) \rightarrow R^1\pi_*(\mathcal {L}\otimes p^*(M))\rightarrow 0 \text { .}
\]
This is the map $u : F_1 \rightarrow F_2$ mentioned in the statement.
For each $L \in U$ we have the exact sequence
\[
N(L) \xrightarrow{u(L)} H^0(C,(L+M+D)\vert _D) \rightarrow H^1(C,L+M) \rightarrow 0 \text { .}
\]

We have $h^0(L+M) \geq 2r+2+t$ if and only if $h^1(L+M)\geq 2r+2+t-d-k+g-1$ if and only if $\rk (u(L)) \leq \deg (D)-2r-2-t+d+k-g+1$.
This is a non-trivial condition if and only if $\rk (N) > \deg (D)-2r-2-t+d+k-g+1$, i.e. $t>0$.
In case $t=0$ then $L+M \in \Pic^{d+k}(C)$ does not belong to a component of type II (it belongs to a component of type III), so we can assume $t>0$.
Also we need $h^1(L+M) \geq 1$ in case $L+M$ belongs to a component of type II.
Therefore we need $t \geq d+k-g-2r=\sum_{i=1}^k e_i -1+2k-2r$ (because of Formula \ref{vgl1}).
Using the splitting sequence $\overrightarrow {e}$ in our situation this implies we need
\[
t \geq (y+1)-(b-1)(y+x) \text { .}
\]
In case $b\geq 2$ then $(y+1)-(b-1)(y+x) \leq (y+1)-(y+x)=1-x$, but $t\geq 1-x$ since $t\geq 1$.
In case $b=1$ we need $t \geq y+1$ but in this case we always have $h^0(L+M)\geq 2r+2+y$ (because equality holds for $L \in \Sigma _{\overrightarrow {e}}(C,f)$) therefore there is no splitting degeneracy locus $\overline{\Sigma}_{\overrightarrow {e'}}(C,f)$ with $\overrightarrow {e'}$ of type $B(0,b',y',t,r+1)$ with $t \leq y$.

In case $L'\in \Sigma _{\overrightarrow {e'}}$ with $\overrightarrow {e'}$ of type $B(0,b',y',t,r+1)$ and $\deg(L')=d+k$ then $L' - M=L$ with $L\in U$ such that $h^0(L+M)=2r+2+t$ and $L \in \Sigma_{\overrightarrow {e'}-1}(C,f)$.
Define $U_t \subset U$ by
\begin{multline*}
U_t = \{ L\in U : h^0(L+M) \geq 2r+2+t \} \\
= \{ L \in U : \rk (u(L)) \leq \deg (D)-2r-2-t+d+k-g+1 \} \text { .}
\end{multline*}
We write $r_t=\deg (D)-2r-2-t+d+k-g+1$.
Of course, by definition $\Sigma _{\overrightarrow {e'}-1}(C,f) \subset U_t$.
On the other hand, if $L \in U_t$ then the splitting sequence $\overrightarrow {e''}$ of $L$ satisfies $e''_i=0$ for $k-r \leq i \leq k$ (because $L \in U$) and $e''_i=-1$ for $k-r-t \leq i \leq k-r-1$.
This clearly implies $\overrightarrow {e''} \leq \overrightarrow {e'}-1$, hence $U_t\subset \overline{\Sigma}_ {\overrightarrow {e'}-1}(C,f)$.
Since $U_t$ is a closed subset of $U$ and $\overline {\Sigma _{\overrightarrow {e'}-1}(C,f)}=\overline {\Sigma}_{\overrightarrow {e'}-1}(C,f)$ we find $\overline {\Sigma}_{\overrightarrow {e'}-1}(C,f) \cap U=U_t$ and $\overline{\Sigma}_{\overrightarrow {e'}-1}(C,f)$ is the closure of $U_t$ in $\Pic ^d(C)$.
Those statements are equivalent to the statement in the proposition except for the fact that as a degeneracy locus of the morphism $u$ one has $\dim (U_t)$ has the expected dimension (if non-empty).

From the previous arguments it follows from Theorem \ref{theorem1} that $\dim (U)-\dim (U_t) = u(\overrightarrow {e'})-u(\overrightarrow {e})$ (and of course $U_t=\emptyset$ if $u(\overrightarrow {e'})>g$).
On the other hand, as a degeneracy locus of the morphism $u$ we have

\begin{multline*}
\dim (U_t) \leq \dim (U) -(\rk (N)-r_t)(\rk (\pi_ *(\mathcal {L}\otimes p^*(M+D)\vert _{U \times D})-r_t) \\
=\dim (U) -t(2r+1+t-d-k+g) \text { .}
\end{multline*}
Therefore we need to show
\[
u(\overrightarrow {e'})-u(\overrightarrow {e})=t(2r+1+t-d-k+g) \text { .}
\]
From Equation \ref{vgl1} we obtain the following two equations
\[
g+k-d-1=by+(b+1)(k-r-1-y)=b(k-r-1)+k-r-1-y
\]

\begin{multline*}
g+k-d-1=(b'+1)y'+(b'+2)(k-r-1-t-y')+t\\
=(b'+1)(k-r-1-t)+k-r-1-t-y'+t
\end{multline*}
and this implies
\[
(b'+1)(k-r-1-t)=b(k-r-1)-y+y' \text { .}
\]
We also have
\[
u(\overrightarrow {e})=[b(k-r-1-y)+(b-1)y](r+1)=b(k-r-1)(r+1)-y(r+1)
\]
and

\begin{multline*}
u(\overrightarrow {e'})=[(b'+1)(k-r-1-t-y')+b'y'](r+1)\\
+[b'(k-r-1-t-y')+(b'-1)y']t\\
=(b'+1)(k-r-1-t)(r+1+t)-(k-r-1-t)t-y'(r+1+t) \text { .}
\end{multline*}
Filling in the previous expression for $(b'+1)(k-r-1-t)$ into $u(\overrightarrow {e'})$ one obtains
\[
u(\overrightarrow {e'})-u(\overrightarrow {e})=((b-1)(k-r-1)-y+t)t=t(g-k-d+t+2r+1) \text { .}
\]
\end{proof}

\begin{remark} \label{rem1}
Using the notations of the previous proof we obtain that if $f$ is general, then locally at a point of $\Sigma_{\overrightarrow {e'}-1}(C,f)$ the set $\overline{\Sigma}_{\overrightarrow {e'}-1}(C,f)$ is given by $\codim _{\overline{\Sigma}_{\overrightarrow {e}}(C,f)}(\overline{\Sigma}_{\overrightarrow {e'}-1}(C,f))$ equations in $\overline{\Sigma}_{\overrightarrow {e}}(C,f)$.
\end{remark}

\begin{remark} \label{rem2}
The proof of Proposition \ref{prop2} shows that for $\overrightarrow {e}$ being of type $B(0,b,y,t,r+1)$ the codimension of $\Sigma _{\overrightarrow {e}}(C,f)$ in $\Pic^d(C)$ is at most $u(\overrightarrow {e})$ without referring to \cite {ref4} as done by explaining Equation \ref{vgl3}.
\end{remark}

In the previous proof we made use of the following lemma which should be well-known.

\begin{lemma}\label{lemma4}
Let $u : E \rightarrow E'$ be a morphism of locally free sheaves on a scheme $T$ of finite type over $K$ and assume for each closed point $t \in T$ the map between the fibers $u(t) : E(t) \rightarrow E'(t)$ is injective.
Then $u$ is a monomorphism and the cokernel $E' / E$ is locally free.
\end{lemma}

\begin{proof}
Viewing $E$ and $E'$ as vectorbundles on $T$ this seems very clear, but here is a direct proof.

Consider the short exact sequences of sheaves
\[
0 \rightarrow \ker (u) \rightarrow E \rightarrow \im(u) \rightarrow 0
\]
\[
0 \rightarrow \im (u) \rightarrow E' \rightarrow \coker (u) \rightarrow 0 \text { .}
\]
For $t \in T$ we have the exact sequence
\[
0 \rightarrow \Tor^{\mathcal{O}_{T,t}}_1(\im(u)_t,K(t)) \rightarrow \ker(u)(t) \rightarrow E(t) \rightarrow \im(u)(t) \rightarrow 0 \text { .}
\]
Since $E(t) \rightarrow \im(u)(t)$ is injective (because $u(t)$ is injective), it is an isomorphism and therefore $\Tor ^{\mathcal{O}_{T,t}}_1(\im (u)_t,K(t)) \rightarrow \ker (u)(t)$ is an isomorphism too.
Since $u(t)$ is an injection it follows also $\im (u)(t) \rightarrow E'(t)$ is an injection.
Therefore, from the exact sequence
\[
0 \rightarrow \Tor^{\mathcal{O}_{T,t}}_1(\coker (u)_t,K(t)) \rightarrow \im (u)(t) \rightarrow E'(t) \rightarrow \coker (u)(t) \rightarrow 0
\]
it follows $\Tor ^{\mathcal{O}_{T,t}}_1(\coker (u)_t,K(t))=0$.
This implies $\coker(u)_t$ is a free $\mathcal{O}_{T,t}$-module for all closed points $t\in T$, hence $\coker (u)$ is locally free on $T$ (see e.g. \cite{ref5}, p. 171 and Corollary A.3.3 p. 617).
But then also $\im (u)_t$ is a free $\mathcal {O}_{T,t}$-module and therefore $\Tor ^{\mathcal {O}_{T,t}}_1(\im (u)_t,K(t))=0$.
This implies $\ker (u) (t)=0$ for all $t \in T$ and using Nakayama's Lemma this implies $\ker (u)=0$.
So we obtain the exact sequence
\[
0 \rightarrow E \rightarrow E' \rightarrow \coker (u) \cong E'/E \rightarrow 0 \text { .}
\]
\end{proof}

\begin{remark}\label{rem3}
In the proof of Proposition \ref{prop2} we use that in case $L \in U$, because of the base-point free pencil trick, we have $h^0(L+M) \geq 2r+2$, hence $L+M \in W^{2r+1}_{d+k}$.
In case $L+M \in \Sigma _{\overrightarrow {e'}}(C,f)$ with $\overrightarrow {e'}$ of type $B(0,b',y',t,r+1)$ with $t\geq 1$ then we have the stronger statement $L+M \in W^{2r+1+t}_{d+k}(C)$.

In case $b' \geq 2$ then we have $L+M$ belongs to a component of type II of $W^{2r+2+t}_{d+k}(C)$.
This means those elements belonging to an open subset of that component are coming from $L \in W^r_d(C)$ being free from $M$ but adding $M$ one gets more global sections then prescribed by the base point free pencil trick.
So from this point of view they have a special property with respect to $M$.
In Proposition \ref{prop2} it is described that those components satisfy a generic behaviour with respect to this special property.

Note that in case $b'=1$ then we can construct $\overrightarrow {e''} \geq \overrightarrow {e'}$ changing $e'_{k-r}=1$ and $e'_{k-r-t-1}=-1$ into $0$ and keeping the other values of $\overrightarrow {e'}$ unchanged.
So we have $\overline{\Sigma} _{\overrightarrow {e'}}(C,f) \varsubsetneq \overline{\Sigma}_{\overrightarrow {e''}}(C,f) \subset W^{2r+1+t}_{d+k}(C)$ and then $\overline{\Sigma}_{\overrightarrow {e'}}(C,f)$ is not an irreducible component of type II (it is properly contained in an irreducible component of type II or III).

In case $\overrightarrow {e'}$ is of type $B(a,b,y,u,v)$ with $a \neq 0$ and $b \geq 2$ (hence $\overline {\Sigma}_{\overrightarrow {e'}}(C,f)$ is a component of type III), then $[L'] \in \Sigma _{\overrightarrow {e'}}(C,f)$ defines $L =L'-aM \in \Sigma _{\overrightarrow {e'}-a}(C,f)$ as described before and $h^0(L')$ is equal to the value one obtains from $h^0(L)$ by applying the base point free pencil trick.
This means $L'$ has no special property with respect to $M$ besides the property already contained in $L$.

This gives the picture that all irreducible components of the schemes $W^r_d(C)$ on a general $k$-gonal curve are in a natural way determined by the irreducible components of $W^r_{d,f}$ satisfying the classical Brill-Noether Theory.
In particular we proved the statements from Theorem A.
\end{remark}

\section{Tangent spaces to the scheme structures}\label{section4}

Up to now we only considered $W^r_d(C)$ and $\overline {\Sigma}_{\overrightarrow {e}}(C,f)$ as sets.
The scheme structure on $W^r_d(C)$ is explained in \cite{ref9}, Chapter IV.
\begin{definition}\label{definition5}
Let $\mathcal{P}_d$ be the Poincar\'e line bundle on $C \times \Pic^d(C)$.
Let $p: C\times \Pic^d(C) \rightarrow C$ and $\pi : C \times \Pic^d(C) \rightarrow \Pic^d(C)$ be the projection morphisms.
Then $W^r_d(C)$ is the subscheme of $\Pic^d(C)$ defined by the $(g-d+r)$-th Fitting ideal of $R^1\pi_*(\mathcal {P}_d)$.
\end{definition}

The Fitting ideal in Definition \ref{definition5} can be described as follows.
Let $D$ be a divisor of large degree on $C$, then we obtain an exact sequence similar to the one occuring in the proof of Proposition \ref{prop2}
\[
\pi_ *(\mathcal{P}_d \otimes \mathcal{O}(p^{-1}(D))) \xrightarrow {u} \pi_*(\mathcal {P}_d \otimes \mathcal {O}(p^{-1}(D)) \vert _{p^{-1}(D)}) \rightarrow R^1\pi_*(\mathcal {P}_d) \rightarrow 0
\]
The $k$-th Fitting ideal of $R^1\pi_*(\mathcal{P}_d)$ is locally generated by the minors of order $\deg (D)-k+1$ using local matrix representations of $u$.

Inspired by this definition and using Formula \ref{vgl2}, in \cite{ref6} one defines a natural scheme structure on $\overline \Sigma _{\overrightarrow {e}}(C,f)$.
We introduce the notation $r(n)+1 = \sum _{i=1}^k \max \{ 0, e_i+n+1 \}$.

\begin{definition}\label{definition6}
We continue to use the notations from Definition \ref{definition5}.
The scheme $\overline{\Sigma}_{\overrightarrow{e}} (C,f)$ is the scheme-theoretical intersection of the subschemes of $\Pic ^d(C)$ defined by the $(g-(d+kn)+r(n))$-th Fitting ideal of $R^1\pi_*(\mathcal{P}_d \otimes p^*(M^{\otimes n}))$ for all $n \in \mathbb{Z}$.
\end{definition}

The theory of Fitting ideals is described in \cite{ref14} (see also \cite{ref5}, 20.2).
In the sequel we will frequently use it commutes with base change (see e.g. \cite{ref5}, Corollary 20.5) without mentioning it.

\begin{remark}\label{rem4}
In Definition \ref{definition6}, using the $n=0$, one finds one of those intersecting subschemes of $Pic^d(C)$ is the scheme $W^{r(0)}_d(C)$ as defined in Definition \ref{definition5}.
In particular we obtain that, as schemes, from the definitions we only know $\overline{\Sigma}_{\overrightarrow {e}}(C,f)$ is a closed subscheme of $W^{r(0)}_d(C)$.
\end{remark}

As already mentioned in Theorem \ref{theorem1}, in \cite{ref7} one proves that, if non-empty, the splitting loci $\Sigma _{\overrightarrow {e}}(C,f)$ in case $f$ is general is smooth as a set.
Moreover in Corollary 1.3 of \cite{ref7} it is stated (together with other statements, see Theorem \ref{theorem3}) that the schemes $W^r_d(C)$ are generically smooth in case $C$ is a general $k$-gonal curve (i.e. it has no multiple components).
We are going to recall the proof of the smoothness of the splitting loci from \cite{ref7} and indicate why the proof of the generically smoothness of the schemes $W^r_d(C)$ is not complete in \cite{ref7}.

\begin{notation}\label{not3}
Let $[L]\in \Pic^d(C)$.
A tangent vector $v \in T_L(\Pic^d(C))$ can be identified with an isomorphism class of an invertible sheaf $\mathcal{L}_v$ of relative degree $d$ on $\pi_{\epsilon} : C \times \Spec (K[\epsilon]) \rightarrow \Spec (K[\epsilon])$ (here $\epsilon ^2=0$) such that for the base change $\iota : \Spec (K) \rightarrow \Spec (K[\epsilon])$ given by the natural epimorphism $K[\epsilon] \rightarrow K[\epsilon]/\langle \epsilon \rangle \cong K$ one has $L \cong \iota^*(\mathcal{L}_v)$.
Using the projection $p_{\epsilon} : C \times \Spec{K[\epsilon}] \rightarrow C$ we have $(p_{\epsilon})_ * (\mathcal{L}_v)$ is a locally free sheaf of rank 2 on $C$ and we also write $\mathcal{L}_v$ to denote this vectorbundle on $C$.
The above mentioned base change corresponds to an epimorphism $\mathcal{L}_v \rightarrow L$ of locally free sheaves on $C$ and we obtain an exact sequence of locally free sheaves on $C$:
\[
0 \rightarrow L \rightarrow \mathcal{L}_v \rightarrow L \rightarrow 0
\]
(using transition functions for $\mathcal {L}_v$ one finds the kernel of the epimorphism has the same transition functions as $L$).
This implies the tangent vector $v$ gives rise to an element of $\Ext^1_C(L,L)$.
This gives a canonical identification of $T_L(\Pic^d(C)$ with $\Ext^1_C(L,L)$.
Using \cite{ref1}, Chapter III, Propositions 6.3 and 6.7, one finds this is isomorphic to $H^1(C,\mathcal{O}_C)$, which is often mentioned as being the tangent space to $\Pic^d(C)$ at some point because $\Pic^0(C)$ is the quotient of $H^1(C,\mathcal{O}_C)$ through the lattice $H^1(C,\mathbb{Z})$.
In Remark \ref{rem6} we make this identification more canonical and we give some comment on the behaviour of this identification under translations by multiples of $M$.

Now let $f : C \rightarrow \mathbb{P}^1$ be a morphism of degree $k$ and write $f_*(L)=\mathcal{O}(\overrightarrow {e})$ for some splitting type $\overrightarrow {e}$.
Let $\mathcal{E}_v=f_* (\mathcal {L}_v)$, a locally free sheaf of rank $2k$ on $\mathbb{P}^1$.
We get the exact sequence
\[
0 \rightarrow E \rightarrow \mathcal{E}_v \rightarrow E \rightarrow 0
\]
and therefore we get an element of $\Ext^1_{\mathbb{P}^1}(E,E)$ we are going to denote by $df(v)$.
Although there is no scheme parametrizing the functor of locally free sheaves of rank $k$ on $\mathbb{P}^1$, this space $\Ext^1_{\mathbb{P}^1}(E,E)$ can be considered as the tangent space of this functor at $E$ (as a matter of fact in the intermediate step of $C \times \Spec(K[\epsilon]) \rightarrow \mathbb{P}^1 \times \Spec(K[\epsilon]) \rightarrow \mathbb {P}^1$ we can consider $\mathcal{E}_v$ as a locally free sheaf of rank $k$ on $\mathbb{P}^1 \times \Spec (K[\epsilon])$ such that $\iota^*(\mathcal{E}_v)\cong E$).
The map $df: \Ext^1_C(L,L) \rightarrow \Ext^1_{\mathbb{P}^1}(E,E)$ is $K$-linear and can be considered as the tangent map at $[L]$ of the morphism from $\Pic^d(C)$ to the functor of locally free sheaves of rank $k$ on $\mathbb{P}^1$.
Again from \cite{ref1} we obtain $\Ext^1_{\mathbb{P}^1}(E,E) \cong H^1(\mathbb{P}^1, E^D \otimes E)$ and this implies $\dim (\Ext^1_{\mathbb{P}^1}(E,E))=u(\overrightarrow {e})$.
\end{notation}

\begin{remark}\label{rem5}
In \cite{ref7} it is proved that in case $f$ is general and $L \in \Sigma_{\overrightarrow {e}}(C,f)$ then the linear map $df$ (see Notation \ref{not3}) is surjective.
Since elements $v \in T_{[L]}(\Pic ^d(C))$ tangent to $\Sigma _{\overrightarrow {e}}(C,f)$ need to be contained in the kernel of $df$ and the codimension of $\Sigma _{\overrightarrow {e}}(C,f)$ in $\Pic ^d(C)$ is at most $u(\overrightarrow {e})$, it follows $\dim (\Sigma _{\overrightarrow {e}}(C,f))=g-u(\overrightarrow {e})$ and $\Sigma _{\overrightarrow {e}}(C,f)$ is smooth as a set.
Note that this argument does not refer to the scheme structure of $\overline {\Sigma}_{\overrightarrow {e}}(C,f)$ or to the scheme structure of $W^r_d(C)$ in case $\overline {\Sigma}_{\overrightarrow {e}}(C,f)$ is an irreducible component of $W^r_d(C)$.
Therefore those arguments are insufficient to prove the statement contained in Corollary 1.3 in \cite{ref7} that $W^r_d(C)$ is generically smooth in case $C$ is a general $k$-gonal curve.

However using the arguments of \cite{ref13}, Section 4, it is easy to see that the results proved in \cite{ref7} imply the schemes $\overline{\Sigma}_{\overrightarrow {e}}(C,f)$ are smooth along $\Sigma _{\overrightarrow {e}}(C,f)$ in case $f$ is general (see Proposition \ref{prop3}; I am grateful to H.K. Larson for mentioning this fact).
However as mentioned in Remark \ref{rem4} it is still insufficient for the statement on $W^r_d(C)$.
\end{remark}

\begin{remark}\label{rem6}
We now make some remarks on the identification between $\Ext^1_C(L,L)$ and $H^1(C,\mathcal{O}_C)$ mentioned in Notation \ref{not3} (and we use those notations).
Using local trivialisations of $L$ we have transitionfunctions $f_{i,j} \in \Gamma (U_i \cap U_j, \mathcal {O}^*_C)$ and there exist $g_{i,j} \in \Gamma (U_i \cap U_j, \mathcal{O}_C)$ such that $f_{i,j}+\epsilon g_{i,j}$ are transitionfunctions for $\mathcal {L}_v$ on $C \times \Spec(K[\epsilon])$.
From the identity $(f_{i,j}+\epsilon g_{i,j})\dot (f_{j,k}+\epsilon g_{j,k})=f_{i,k}+\epsilon g_{i,k}$ we obtain $g_{i,k}=g_{i,j}f_{j,k}+f_{i,j}g_{j,k}$.
This implies $\frac{g_{i,k}}{f_{i,k}}=\frac{g_{i,j}}{f_{i,j}}+\frac{g_{j,k}}{f_{j,k}}$, therefore $\frac{g_{i,j}}{f_{i,j}} \in \Gamma (U_i \cap U_j, \mathcal {O}_C)$ satisfies the cocyle condition and one obtains $c(v) \in H^1(C,\mathcal{O}_C))$.
This is an explicit description of the isomorphism $c : \Ext^1_C(L,L) \rightarrow H^1(C,\mathcal {O}_C)$.

Now let $m_{i,j}$ be transition functions of $M$ (the line bundle associated to $f$).
For each $n \in \mathbb{Z}$ we have an extension
\[
0 \rightarrow L \otimes M^{\otimes n} \rightarrow \mathcal {L}_v \otimes M^{\otimes n} \rightarrow L \otimes M^{\otimes n} \rightarrow 0
\]
corresponding to $v_n \in \Ext^1_C(L + nM,L + nM)=T_{L+nM}(\Pic^{d+kn}(C))$.
As described above we have the explicit description of the identification $c_n : \Ext^1_C(L + nM,L + nM)  \rightarrow H^1(C,\mathcal{O}_C)$.
Since the transition functions of $L + nM$ are $f_{i,j}m_{i,j}^n$ and the transition functions of $\mathcal{L}_v\otimes M^{\otimes n}$ are $f_{i,j}m_{i,j}^n+g_{i,j}m_{i,j}^n$ it follows $c_n(v_n)$ is also described by the cocycle $\frac{g_{i,j}}{f_{i,j}}$.
This shows that the identifications $c$ and $c_n$ are compatible with the tangent map of the translation $\Pic^d(C) \rightarrow \Pic^{d+nk}(C) : L \rightarrow L+nM$.
\end{remark}

We show the arguments of a statement already mentioned in Remark \ref{rem5}.
\begin{proposition}\label{prop3}
Let $f : C \rightarrow \mathbb{P}^1$ be a general morphism of degree $k$ with $C$ a smooth curve of genus $g$.
In case $\Sigma _{\overrightarrow {e}}(C,f)$ is not empty (i.e. $u(\overrightarrow {e})\geq g$) then the scheme $\overline{\Sigma}_{\overrightarrow {e}}(C,f)$ is smooth along $\Sigma _{\overrightarrow {e}}(C,f)$.
\end{proposition}

To make the connection between the line bundles on $C$ and the locally free sheaves on $\mathbb{P}^1$ we prove the following lemma.
\begin{lemma}\label{lemma5}
Let $f :C \rightarrow \mathbb{P}^1$ be a morphism of degree $k$ with $C$ a smooth curve of genus $g$.
We use Notation\ref{not1} using the Poincar\'e sheaf and we also write $f$ to denote $f_{\Pic^d(C)}$.
We denote both projections $C \times \Pic^d(C) \rightarrow \Pic^d(C)$ and $\mathbb{P}^1 \times \Pic^d(C) \rightarrow \Pic^d(C)$ by $\pi$.
For $n$ we write $\mathcal{P}_d(n)$ to denote $\mathcal{P}_d \otimes f^*(\mathcal {O}_{\mathbb{P}^1}(n))$ and $E_d(n)$ to denote $E_d \otimes \mathcal{O}_{\mathbb{P}^1}(n)$.
For any $i\geq 0$ and any $n \in \mathbb{Z}$ the $i$-th Fitting ideals of $R^1\pi_*(\mathcal{P}_d(n))$ and $R^1\pi_*(E_d(n))$ are equal.
\end{lemma}
\begin{proof}
Both projection maps $\pi$ are related by composition with $f : C \times \Pic^d(C) \rightarrow \mathbb{P}^1 \times \Pic^d(C)$.
Since $R^1f_*(\mathcal{P}_d)=0$ it follows from the Leray spectral sequence (see e.g. \cite {ref10}, Chapter III, Theorem 1.18) that $R^1\pi_*(\mathcal {P}_d(n))\cong R^1\pi_*(f_*(\mathcal {P}_d(n)) \cong R^1\pi_*(E_d(n))$.
In the second isomorphism we use the projection formula implying $f_*(\mathcal{P}_d(n)) \cong f_*(\mathcal{P}_d) \otimes \mathcal{O}_{\mathbb{P}^1}(n)$.
\end{proof}
\begin{proof}[Proof of Proposition \ref{prop3}]
In Lemma 4.1 of \cite{ref13} it is proved that the tangent space to the scheme defined as the intersection of the schemes defined by the $(g-(d+kn)+r(n))$-th Fitting ideal of $R^1\pi_*(E_d(n))$ with $n \in \mathbb{Z}$ is equal to the kernel of $df: \Ext^1_C(L,L) \rightarrow \Ext^1_{\mathbb{P}^1}(E,E)$ in case $[L] \in \Sigma _{\overrightarrow {e}}(C,f)$.
Because of Lemma \ref{lemma5} and the arguments recalled in Remark \ref{rem5} this finishes the proof.
\end{proof}

The proof of Lemma 4.1 in \cite{ref13} is based on Lemma 4.2 in \cite{ref13}.
In order to finish the proof that $W^r_d(C)$ is generically smooth in case $C$ is a general $k$-gonal curve, we need to replace this lemma by a more old fashioned-type lemma.

\begin{lemma}\label{lemma6}
We continue to use the situation from Lemma \ref{lemma5}, we extend the use of $(n)$ and we also use Notation \ref{not3}.
Let $[L]\in \Sigma _{\overrightarrow {e}}(C,f)$ and $v \in T_{[L]}(\Pic^d (C))$.
This tangent vector $v$ belongs to the tangent space of the $(g-(d+kn)+r(n))$-th Fitting ideal of $R^1\pi_*(E_d(n))$ if and only if $H^0 (\mathbb{P}^1,\mathcal{E}_v(n)) \rightarrow H^0(\mathbb{P}^1,E(n))$ is surjective.
\end{lemma}

This lemma should be compared to the following fact on the schemes $W^r_d(C)$.
Let $[L]\in W^r_d(C) \setminus W^{r+1}_d(C)$.
In \cite{ref9}, Chapter IV, Proposition (4.2), it is proved that $v \in T_{[L]}(\Pic^d(C))$ it a tangent vector to the scheme $W^r_d(C)$ if and only if $H^0(C, \mathcal{L}_v) \rightarrow H^0(C,L)$ is surjective.

\begin{proof}[Proof of Lemma \ref{lemma6}]
By replacing $L$ with $L \otimes M^{\otimes n}$ we can assume $n=0$.
We write $m$ to denote $g-d+r(0)$.
Let $D$ be an effective divisor of very large degree on $\mathbb{P}^1$ and consider the exact sequence on $\mathbb{P}^1 \times \Pic^d(C)$
\[
0 \rightarrow E_d \rightarrow E_d(D) \rightarrow E_d(D)\vert D \rightarrow 0
\]
(using some short notations omitting the projection $p : \mathbb{P}^1 \times \Pic^d(C) \rightarrow \mathbb{P}^1$).
It gives rise to an exact sequence on $\Pic ^d(C))$:
\[
\pi_*(E_d(D)) \xrightarrow {u}\pi_*(E_d(D)\vert _D) \rightarrow R^1\pi_*(E_d) \rightarrow 0 \text { .}
\]
Since $R^1\pi_*(E_d)([L])\cong H^1(\mathbb{P}^1,E)$ we obtain the exact sequence
\[
0 \rightarrow H^0(\mathbb{P}^1,E) \rightarrow H^0(\mathbb{P}^1,E(D)) \xrightarrow {u([L])} H^0(\mathbb{P}^1, E(D) \vert _D) \rightarrow H^1(\mathbb{P}^1, E) \rightarrow 0 \text { .}
\]

For $v \in T_{[L]}(\Pic^d(C))$, taking base change with the associated morphism $\Spec (K[\epsilon]) \rightarrow \Pic^d(C)$, we get the exact sequence
\[
\pi_*(\mathcal{E}_v(D)) \xrightarrow {u_v} \pi_*(\mathcal {E}_v(D) \vert _D) \rightarrow R^1\pi_*(\mathcal {E}_v) \rightarrow 0
\]
and pulling back by $\iota : [L]=\Spec (K) \rightarrow \Spec (K[\epsilon])$ we obtain the above exact sequence.
(Above, we also write $\pi$ to denote the projection map $\mathbb{P}^1 \times \Spec(K[\epsilon]) \rightarrow \Spec (K[\epsilon])$.)
By definition $v$ is tangent to the scheme defined by the $m$-th Fitting ideal of $R^1\pi_*(E_d)$ if and only if all $(m+1)\times (m+1)$-minors of a matrix representation of $u_v$ are zero.
This matrix representation is a $\sum_{i=1}^k(e_i+\deg (D)+1)\times k\deg(D)$-matrix over $K[\epsilon]$.

In case $H^0(\mathbb{P}^1, \mathcal{E}_v) \rightarrow H^0(\mathbb{P}^1,E)$ is surjective then, using liftings of a $K$-base of $H^0(\mathbb{P}^1,E)$ to $H^0(\mathbb{P}^1,\mathcal{E}_v)$, we find we can take a $K[\epsilon]$-base for $\pi_* (\mathcal{E}_v(D))$ and $\pi _*(\mathcal {E}_v(D) \vert _D)$ such that the first $h^0(E)=r+1$ columns of the matrix representation on $u_v$ are zero.
This holds because those liftings are the $K[\epsilon]$-base of a free $K[\epsilon]$-submodule of $\pi_*(\mathcal{E}_v)=\ker (u_v)$ (this can be considered as an application of Lemma \ref{lemma4} using $T =\Spec (K[\epsilon])$).
This implies all minors of order $(\sum_{i=1}^ke_i)+k+k\deg(D)-r=d-g-r+k\deg(D)+1$ are zero, hence $v$ belongs to the tangent space of the scheme defined by the $m$-th Fitting ideal of $R^1\pi_*(E_d)$.

Now assume $v$ belongs  to the tangent space of the scheme defined by the $m$-th Fitting ideal of $R^1\pi_*(E_d)$.
Using the exact sequence with $u([L])$ we obtain we can find bases over $K[\epsilon]$ for $\pi_*(\mathcal {E}_v(D))$ and $\pi_*(\mathcal{E}_v(D) \vert _D)$ such that the matrix corresponding to $u_v$ has reduction modulo $\langle \epsilon \rangle$ with the first $r+1$ columns equal to zero, the left upper square of order $d-g+k\deg(D)-r$ equal to the identity matrix and the last $g-d+r$ rows equal to zero.
So the first $r+1$ base elements of $\pi_*(\mathcal {E}_v(D))$ correspond to a base of $h^0(\mathbb{P}^1,E)$.

The minor of order $d-g+k\deg(D)-r+1$ of that matrix corresponding to $u_v$ containing the upper square matrix of order $d-g+k\deg(D)-r$ and the $(i,j)$-th element for some $i \leq r+1$ and $j> d-g+k\deg (D)-r$ has to be zero.
This implies this $(i,j)$-th element has to be zero.
It follows that for $1 \leq i \leq r+1$, changing the $i$-th base element of $\pi_*(\mathcal{E}_v(D))$ by adding a suitable linear combination using coefficients whose reduction modulo $\langle \epsilon \rangle$ are 0 of the $(r+1+j)$-th base elements with $1 \leq j \leq d-g+k\deg(D)-r$, we obtain the $i$-th column of $u_v$ becomes zero.
Those new base elements correspond to global sections of $\pi_*(\mathcal{E}_v(D))$ being lifts of elements of a base of $H^0(\mathbb{P}^1,E)$ and belonging to the kernel of $H^0(u_v)$, hence belonging to $H^0(\mathbb{P}^1, \mathcal{E}_v)$.
This proves the reduction map $H^0(\mathbb{P}^1,\mathcal{E}_v) \rightarrow H^0(\mathbb{P}^1, E)$ is surjective.
\end{proof}

Using Lemma \ref{lemma6} one can obtain a proof of Lemma 4.1 in \cite{ref13} as follows.
From Lemma \ref{lemma6} it follows $v \in T_{[L]}\Pic^d(C)$ with $[L] \in \Sigma _{\overrightarrow {e}}(C,f)$ is tangent to the scheme $\overline{\Sigma} _{\overrightarrow {e}}(C,f)$ if and only if for all $n \in \mathbb{Z}$ the map $H^0(\mathbb{P}^1, \mathcal{E}_v(n)) \rightarrow H^0(\mathbb{P}^1,E(n))$ is surjective.
Now we can use the following lemma that should be well-known.

\begin{lemma}\label{lemma7}
Let $0 \rightarrow E_1 \rightarrow E_2 \rightarrow E_3 \rightarrow 0$ be an exact sequence of locally free sheaves on $\mathbb{P}^1$.
This sequence splits if and only if for each $n \in \mathbb{Z}$ the associated map $H^0(\mathbb{P}^1,E_2(n)) \rightarrow H^0(\mathbb{P}^1,E_3(n))$ is surjective.
\end{lemma}
\begin{proof}
The only if statement is trivial.

Assume for all $n \in \mathbb{Z}$ the associated map $H^0(\mathbb{P}^1,E_2(n)) \rightarrow H^0(\mathbb{P}^1,E_3(n))$ is surjective.
Let $n \in \mathbb{Z}$ be such that the locally free sheaf $E_3(n)$ has a direct summand $\mathcal{O}_{\mathbb{P}^1}$.
The global section $1 \in H^0(\mathbb{P}^1,\mathcal{O}_{\mathbb{P}^1})$ of that summand can be lifted to a global section of $E_2(n)$ (because of the surjectivity).
This gives rise to an injection $\mathcal {O}_{\mathbb{P}^1} \rightarrow E_2(n)$, injective on the fibers, such that the composition of the morphism to $E_3(n)$ and the projection on that component is the identity.
Using this one can make stepwise a section $E_3 \rightarrow E_2$ of the short exact seqeuence.
\end{proof}

In order to prove the generic smoothness of the schemes $W^r_d(C)$ in case $f$ is general, we first mention an easy reduction of the number of Fitting ideals we have to consider to obtain the scheme-structure in Definition \ref{definition6}.

\begin{lemma}\label{lemma8}
In order to define the scheme structure on $\overline{\Sigma}_{\overrightarrow {e}}(C,f)$ as defined in Definition \ref{definition6} we can take $-e_k \leq n < -e_1-1$.
\end{lemma}
\begin{proof}
Let $L \in \Pic^d(C)$.

Let $n<-e_k$.
Then $u(L): \pi_*(\mathcal{P}_d(n)(D))(L)=H^0(C, L + nM +D)\rightarrow \pi_*(\mathcal{P}_d(n)(D)\vert_D)(L)=H^0(C,L + nM +D\vert_D)$, hence $\ker (u(L))=H^0(C, L + nM)=0$.
From Lemma \ref{lemma4} it follows that locally at $L$ the sheaf $R^1\pi_*(\mathcal{P}_d(n))$ is locally free.
Therefore locally at $L$ we can use a presentation $0 \rightarrow 0 \rightarrow R^1\pi_*(\mathcal{P}_d(n)) \rightarrow R^1\pi_(\mathcal{P}_d(n)) \rightarrow 0$.
This shows the Fitting ideal is trivial locally at $L$.

Let $n \geq -(e_1+1)$.
For $L \in \Sigma _{\overrightarrow {e}}(C,f)$ we have $h^0(L+(n+1)M)-h^0(L+nM)=k$.
Since $\deg (M)=k$ this implies $L + nM$ is non-special (see the argument in the proof of Lemma \ref{lemma3}), in particular $h^1(L+nM)=0$.
This implies $R^1\pi_*(\mathcal{P}_d(n))=0$ in a neighbourhood of $L$ and then of course the Fitting ideal of $R^1\pi_*(\mathcal{P}_d(n))$ is trivial locally at $L$.
\end{proof}

The following proposition gives a further number of reductions on the integers $n \in \mathbb{Z}$ we have to consider in order to describe the tangent space of $\overline {\Sigma}_{\overrightarrow {e}}(C,f)$ at $L\in \Sigma_{\overrightarrow {e}}(C,f)$.

\begin{proposition}\label{prop4}
Let $f:C \rightarrow \mathbb{P}^1$ be a morphism of degree $k$ for a smooth curve $C$ of genus $g$ and let $L \in \Sigma _{\overrightarrow {e}}(C,f)$.
The tangent space to the scheme $\overline{\Sigma}_{\overrightarrow {e}}(C,f)$ at $L$ is equal to the intersection of the tangent spaces at $L$ of the subschemes of $\Pic^d(C)$ defined by the $(g-(d-ke_i)+r(-e_i))$-th Fitting ideals of $R^1\pi_*(\mathcal{P}_d(-e_i))$ with $e_1+1<e_i \leq e_k$.
\end{proposition}
\begin{proof}
Assume $e_{i-1} < e_i-1$.
It is enough to prove that for each $e_{i-1} <e' \leq e_i-1$ the tangent space at $L$ to the scheme defined by the $(g-(d-ke')+r(-e'))$-th Fitting ideal of $R^1\pi_*(\mathcal{P}_d(-e'))$ contains the tangent space at $L$ to the scheme defined by the $(g-(d-k(e'+1))-r(-e'-1))$-th Fitting ideal of $R^1\pi_*(\mathcal{P}_d(-e'-1))$.

Using the isomorphism $c: \Ext^1(L,L) \rightarrow H^1(C,\mathcal{O}_C)$ described in Remark \ref{rem6} one obtains the following description for the tangent space to the scheme defined by the $(g-(d-ke')-r(-e'))$-th Fitting ideal of $R^1\pi_*(\mathcal{P}_d(-e'))$ at $L$ in \cite{ref9}, Chapter IV, Proposition 4.2.
Consider the multiplication map
\[
\mu_1 : H^0(C,L-e'M) \otimes H^0(C,K - (L-e'M)) \rightarrow H^0(C,K)
\]
then the tangent space is the orthogonal complement of $\im (\mu_1) \subset H^0(C,K)$ under the Serre duality $H^0(C,K)\otimes H^1(C,\mathcal{O}_C) \rightarrow K$.
As mentioned in Remark \ref{rem6} the isomorphism $c$ is compatible with the translation by multiples of $M$ under the various connected components of $\Pic (C)$.
In a similar way we use
\[
\mu_2 : H^0(C,L-(e'+1)M) \otimes H^0(C, K - (L-(e'+1)M)) \rightarrow H^0(C,K)
\]
and we need to prove $\im (\mu _1)\subset \im (\mu _2)$.

Consider the multiplication map
\[
m_1 : H^0(C,L -(e'+1)M)\otimes H^0(C,M) \rightarrow H^0(C,L -e'M)
\]
Because of the base point free pencil trick (see \cite {ref9}, p. 126) we have $\ker (m_1)\cong H^0(C,L-(e'+2)M)$.
Let $N=h^0(L-(e'+1)M))$.
We have $e_i-(e'+2) \geq -1$ and $e_{i-1}-(e'+2)<-2$ and therefore $h^0(L-(e'+2)M)=N-(k-i+1)$.
So we find $\dim ( \im (m_1))\geq 2N-N+(k-i+1)=N+(k-i+1)$.
Since $e_i - e' \geq 1$ and $e_{i-1}-e' >0$ we also have $h^0(L-e'M)=N+k-i+1$.
This proves $m_2$ is surjective.

Let $a_1, \cdots ,a_N$ be a base for $H^0(C,L-(e'+1)M)$ and let $s,t$ be a base for $H^0(C,M)$.
From the surjectivity of $m_1$ it follows the elements $sa_1, \cdots, sa_N,\linebreak ta_1, \cdots ,ta_N$ generate $H^0(C,L-e'M)$.
Therefore each element of $\im (\mu_1)$ can be written as $\sum_{i=1}^N(sa_ix_i+tb_iy_i)$ with $x_i, y_i \in H^0(C,K- (L-e'M))$.
But $\sum_{i=1}^N(sa_ix_i+tb_iy_i)=\mu_2(\sum_{i=1}^Na_i\otimes (sx_i+ty_i)) \in \im (\mu _2)$.
This proves $\im (\mu_1 )\subset \im (\mu _2)$.
\end{proof}

\begin{corollary}\label{cor1}
Let $C$ be a general $k$-gonal curve and let $Z$ be an irreducible component of type I of some scheme $W^r_d(C)$ then the scheme $W^r_d(C)$ is smooth at a general point of $Z$.
\end{corollary}

\begin{proof}
We know $Z=\overline{\Sigma}_{\overrightarrow {e}}(C,f)$ with $\overline {e}$ of type $B(0,b,y,r+1,0)$ for some $b \geq 1$.
We know $\overline{\Sigma}_{\overrightarrow {e}}(C,f)$ is the closure of $\Sigma_{\overrightarrow {e}}(C,f)$ and the scheme $\overline {\Sigma}_{\overrightarrow {e}}(C,f)$ is smooth at $L \in \Sigma _{\overrightarrow {e}}(C,f)$ because of Proposition \ref{prop3}.
From Proposition \ref{prop4} and taking into account Lemma \ref{lemma8} we find that the tangent space to $\overline{\Sigma}_{\overrightarrow{e}}(C,f)$ at $L \in \Sigma_{\overrightarrow {e}}(C,f)$ is exactly the tangent space to the scheme defined by the $(g-d+r)$-th Fitting ideal of $R^1\pi_*(\mathcal{P}_d)$.
From Definition \ref{definition5} we know this is the tangent space to the scheme $W^r_d(C)$ at $L$ and we obtain $\dim (T_{L}(W^r_d(C))=\dim (Z) (=\rho^r_d(g))$.
It follows the scheme $W^r_d(C)$ is smooth at $L$.
\end{proof}

\begin{proposition}\label{prop5}
Let $C$ be a general $k$-gonal curve and let $Z$ be an irreducible component of type II of some scheme $W^r_d(C)$ then the scheme $W^r_d(C)$ is smooth at a general point of $Z$.
\end{proposition}

\begin{proof}
We have $Z = \overline{\Sigma}_{\overrightarrow {e}}(C,f)$ with $\overrightarrow {e}$ of type $B(0,b,y,u,v)$ with $v>0$, $u>0$, $u+2v = r+1$ and $b>1$ and such that Formula \ref{vgl1} holds.
Let $L\in \Sigma _{\overrightarrow {e}}(C,f)$.
Because of Proposition \ref{prop4} the tangent space of the scheme $Z$ at $L$ is equal to the intersection of the tangent spaces of the schemes defined by the $(g-(d+nk)+r(n))$-th Fitting ideal of $R^1\pi_*(\mathcal{P}_d(n))$ with $n \geq 0$.
Therefore intersecting the tangent space of $W^r_d(C)$ at $L$ with the tangent space to the scheme defined by the $(g-(d-k)+v-1)$-th Fitting ideal of $R^1\pi_*(\mathcal{P}_d(-1))$ we obtain the tangent space of $\overline {\Sigma}_{\overrightarrow {e}}(C,f)$ at $L$.
Because of Proposition \ref{prop3} this tangent space has dimension equal to $\dim (Z)$.
Therefore it is sufficient to prove that each tangent vector $v \in T_{L}(W^r_d(C))$ is a tangent vector to the scheme defined by the $(g-(d-k)+v-1)$-th Fitting ideal of $R^1\pi_*(\mathcal{P}_d(-1))$.
Because of Lemma \ref{lemma6} we need to prove the following statement.
Let $v \in T_{L}(\Pic^d(C))=\Ext^1(L,L)$ and let $0 \rightarrow E \rightarrow \mathcal{E}_v \rightarrow E \rightarrow 0$ be the extension of locally free sheaves on $\mathbb{P}^1$ corresponding to $df(v)$ (see Notation \ref{not3}).
If $H^0(\mathbb{P}^1,\mathcal{E}_v) \rightarrow H^0(\mathbb{P}^1,E)$ is surjective then $H^0(\mathbb{P}^1,\mathcal{E}_v(-1)) \rightarrow H^0(\mathbb{P}^1,E(-1))$ is surjective.

Since $\mathcal{E}_v$ is a locally free sheaf of rank $2k$ on $\mathbb{P}^1$ it is isomorphic to $\mathcal{O}(\overrightarrow {e'})$ for some splitting sequence $\overrightarrow {e'}=(e'_1 \cdots , e'_{2k})$ such that $\sum_{i=1}^{2k}e'_i=2 \sum_{i=1}^k e_i$ (because $\Lambda^{2k}\mathcal{E}_v\cong \Lambda ^k(E) \otimes \Lambda ^k(E)$).
It is enough to prove that if $H^0(\mathbb{P}^1,\mathcal{E}_v) \rightarrow H^0(\mathbb{P}^1,E)$ is surjective then $\overrightarrow {e'}$ if of type $B(0,b,2y,2u,2v)$ (in that case $\mathcal{E}_v \cong E \oplus E$).

$\underline {\text{Claim 1}}$: $e'_1 \geq -b-1$.

We have an exact sequence $H^1(\mathbb{P}^1, E(b)) \rightarrow H^1 (\mathbb{P}^1, \mathcal {E}_v(b)) \rightarrow H^1(\mathbb{P}^1, E(b))$.
Since $H^1(\mathbb{P}^1,E(b))=0$ this implies $H^1(\mathbb{P}^1,\mathcal{E}_v(b))=0$.
This implies $-e'_1-b-2<0$, hence $e'_1 >-b-2$.

$\underline {\text{Claim 2}}$: If $e'_{x'}=-b-1$ and $e'_{x'+1}\geq -b$ then $x' \leq 2x$ with $x=k-y-u-v$.
(Of course, if $e'_1 \geq -b$ then we put $x'=0$ and there is nothing to prove.)

In the exact sequence $H^1(\mathbb{P}^1,E(b-1))\rightarrow H^1(\mathbb{P}^1,\mathcal{E}_v(b-1)) \rightarrow H^1(\mathbb{P}^1,E(b-1)) \rightarrow 0$ we have $\dim (H^1(E(b-1)))=x$ and because of Claim 1 we have $\dim (H^1(\mathbb{P}^1,\mathcal{E}_v(b-1)))=x'$.
This proves Claim 2.

$\underline {\text{Claim 3}}$: $e'_{2k} \leq 1$.

We have the exact sequence $0 \rightarrow H^0(\mathbb{P}^1,E(-2))\rightarrow H^0(\mathbb{P}^1,\mathcal{E}_v(-2)) \rightarrow H^0(\mathbb{P}^1,E(-2))$.
Since $H^0(\mathbb{P}^1,E(-2))=0$ this implies $H^0(\mathbb{P}^1,\mathcal{E}_v(-2))=0$
This implies $e'_{2k}-2<0$, hence $e'_{2k}<2$.

$\underline{ \text{Claim 4}}$: Let $v'$ be the number of elements $e'_i$ equal to 1, then $v' \leq 2v$.

We have the exact sequence $0 \rightarrow H^0(\mathbb{P}^1,E(-1)) \rightarrow H^0(\mathbb{P}^1,\mathcal{E}_v(-1)) \rightarrow H^0(\mathbb{P}^1,E(-1))$ with $\dim (H^0 (\mathbb{P}^1,E(-1))=v$ and because of Claim 3 we have $\dim (H^0 (\mathbb{P}^1,\mathcal{E}_v(-1)=v'$.
This proves Claim 4.

Suppose $0 \rightarrow H^0(\mathbb{P}^1,E) \rightarrow H^0(\mathbb{P}^1,\mathcal{E}_v) \rightarrow H^0(\mathbb{P}^1,E) \rightarrow 0$ is exact.
We write $u'$ to denote the number of elements $e'_i$ equal to 0.
We have $\dim (H^0(\mathbb{P}^1,E))=u+2v$ and because of Claim 3 we have $\dim (H^0(\mathbb{P}^1,\mathcal{E}_v))=u'+2v'$.
This implies
\[
u'+2v'=2u+4v \text { .}
\]

Let $v'=2v$, then also $u'=2u$.
Since $\sum_{i=1}^{2k}e'_i=2\sum_{i=1}^ke_i$ we need $\sum _{i=1}^{2x+2y}e'_i=2x(-b-1)+2y(-b)$.
Since $e'_i=-b-1$ for $1 \leq i \leq x'$ with $x'\leq 2x$ (because of Claims 1 and 2) and $e'_i\geq -b$ for $x'+1 \leq i \leq 2x+2y$, we need $x'=2x$ and $e'_i=-b$ for $2x+1 \leq i \leq 2x+2y$.
So in this case we obtain $\mathcal{E}_v\cong E \oplus E$.

In case $v'\neq 2v$ then because of Claim 4 we have $v'<2v$ and $u'=2u+2(2v-v')$.
From $\sum_{i=1}^{2k}e'_i=2 \sum_{i=1}^k e_i$ we obtain
\[
\sum_{i=1}^{2k-v'+1}e'_i = 2x(-b-1)+2y(-b)+(2v-v') \text { .}
\]
Because of Claim 1 we also have
\[
\sum_{i=1}^{2k-v'+1}e'_i \geq x'(-b-1)+(2x-x'+2y-(2v-v'))(-b) \text { .}
\]
This implies
\[
2v-v' \geq (2x-x')+(2v-v')b \text { .}
\]
Since $2x \geq x'$ and $b\geq 2$ this is impossible.
It follows that $v'\neq 2v$ cannot occur, finishing the proof.
\end{proof}

In order to prove that irreducible components of type III of $W^r_d(C)$ are generically smooth in case $C$ is a general $k$-gonal curve, first we have to sharpen the statement of Proposition \ref{prop4} in relevant cases.

\begin{proposition}\label{prop6}
Let $f:C \rightarrow \mathbb{P}^1$ be a morphism of degree $k$ with $C$ a smooth curve of genus $C$.
Let $L \in \Sigma_{\overrightarrow {e}}(C,f)$ and let $e' \in \mathbb{Z}$ such that for some $i \geq 2$ we have $e_{i-1}+2 \leq e' < e_i$.
The tangent space at $L$ to the scheme defined by the $(g-(d-ke')+r(-e'))$-th Fitting ideal of $R^1\pi_*(\mathcal{P}_d(-e'))$ is equal to the tangent space of the scheme defined by the $(g-(d-ke_i)+r(-e_i))$-th Fitting ideal of $R^1\pi_*(\mathcal{P}_d(-e_i))$.
\end{proposition}

\begin{proof}
It is enough to prove that the tangent space at $L$ of the scheme defined by the $(g-(d-ke')+r(-e'))$-th Fitting ideal of $R^1\pi_*(\mathcal{P}_d(-e'))$ is equal to the tangent space of the scheme defined by the $(g-(d-k(e'+1)+r(-e'-1))$-th Fitting ideal of $R^1\pi_*(\mathcal{P}_d(-e'-1))$.
We continue to use the arguments and notations of the proof of Proposition \ref{prop4}; in particular it is enough to prove $\dim (\im (\mu_1))=\dim (\im (\mu_2))$.

Since we assume $e'\geq e_{i-1}+2$ we can use the base point free pencil trick to prove that also the following multiplication map is surjective
\[
m_2 : H^0(C,K - (L -e'M))\otimes H^0(C,M) \rightarrow H^0(C,K -(L -(e'+1)M)) \text { .}
\]
Indeed, because of the base point free pencil trick we have
\begin{multline*}
\dim (\ker m_2)=\dim (H^0(C,K - (L -(e'-1)M))\\
=N+2(k-i+1)-(d-k(e'-1)+g-1) \text { .}
\end{multline*}
Since $h^0(K-L+e'M)=N+(k-i+1)-(d-ke'+g-1)$ we find $\dim (\im (m_2))=N-(d-k(e'+1)+g-1)=h^0(K-L+(e'+1)M)$.

Now we use the multiplication map
\[
\mu'_1 : (H^0(C,L- (e'+1)M)\otimes H^0(C,M)) \otimes H^0(C,K - (L -e'M))\rightarrow H^0(C,K) \text { .}
\]
From the surjectivity of $m_1$ it follows that $\im (\mu'_1)=\im (\mu_1)$.
Taking into account the results of Proposition \ref{prop5}, we have to prove that $\dim (\im (\mu'_1))=\dim (\im (\mu_2))$.

As in the proof of Proposition \ref{prop5}, let $a_1, \cdots , a_N$ be  a base for $H^0(C,L -(e'+1)M)$ and let $s,t$ be a base for $H^0(C,M)$.
Let $\sum_{i=1}^N((a_i\otimes s)\otimes x_i+(a_i \otimes t) \otimes y_i) \in \ker (\mu'_1)$ with $x_i,y_i \in H^0(C,K -(L - e' M))$.
Then of course $\sum_{i=1}^Na_i \otimes (sx_i+ty_i) \in \ker (\mu_2)$, so we obtain a linear map $\mu : \ker (\mu'_1) \rightarrow \ker (\mu _2)$.

Let $\sum_{i=1}^N a_i \otimes z_i \in \ker (\mu_2)$ with $z_i \in H^0(C,K-(L-(e'+1)M))$.
Because $m_2$ is surjective there exist $x'_i,y'_i \in H^0(C,K -(L -e'M))$ such that $z_i=x'_is+y'_it$.
We obtain $\sum_{i=1}^N a_i \otimes z_i =\mu (\sum_{i=1}^N((a_i \otimes s)\otimes x'_i +(a_i \otimes t) \otimes y'_i))$ proving that $\mu$ is surjective.

Let $\sum_{i=1}^N ((a_i \otimes s) \otimes x_i+(a_i \otimes t) \otimes y_i) \in \ker (\mu)$ with $x_i, y_i \in H^0(C,K -(L - e'M))$.
We have $\sum_{i=1}^N a_i\otimes (sx_i+ty_i) = 0$ in $H^0 (C,L -(e'+1)M)\otimes H^0(C,K -(L-(e'+1)M))$.
This implies for $1 \leq i \leq N$ we have $sx_i+ty_i=0$ in $H^0(C,K -(L -(e'+1)M))$.
Using once more the base point free pencil trick, this is equivalent to the existence of $z_i \in H^0(C,K -(L- (e'-1) M))$ with $x_i=tz_i$ and $y_i=-sz_i$.
It follows
\[
\sum_{i=1}^N (a_i\otimes s \otimes x_i + a_i \otimes t \otimes y_i) = \sum _{i=1}^N(a_i \otimes s \otimes tz_i-a_i\otimes t \otimes sz_i) \text { .}
\]
This proves
\begin{multline*}
\dim (\ker \mu)=N \cdot \dim H^0(C,K -(L-(e'-1)M))\\
=N \cdot (N+2(k-i+1)-(d-k(e'-1)-g+1)) \text { .}
\end{multline*}

We find
\[
\dim (\ker (\mu_2))= \dim (\ker (\mu'_1))-N\cdot (N+2(k-i+1)-(d-k(e'-1)-g+1)) \text { .}
\]
This implies
\begin{multline*}
\dim (\im \mu_2))=N\cdot (N-(d-k(e'+1)-g+1))-\dim (\ker (\mu_2))\\
=N\cdot (N-(d-k(e'+1)-g+1))-\dim (\ker (\mu'_1))+N\cdot (N+2(k-i+1)-(d-k(e'-1)-g+1))\\
=2N\cdot (N+(k-i+1)-(d-ke'-g+1))-\dim (\ker (\mu'_1))=\dim (\im (\mu'_1)) \text { .}
\end{multline*}
\end{proof}

\begin{corollary}\label{cor2}
Let $C$ be a general $k$-gonal curve and let $Z$ be an irreducible component of type III of some scheme $W^r_d(C)$ then the scheme $W^r_d(C)$ is smooth at a general point of $Z$.
\end{corollary}

\begin{proof}
We have $Z=Z'+e'M$ with $Z' =\overline{\Sigma}_{\overrightarrow{e}}(C,f)$ and $\overrightarrow {e}$ of type $B(0,b,y,u,v)$ such that $b+e' \leq -2$.
Let $L = L' +e' M$ with $L'\in \Sigma _{\overrightarrow {e}}(C,f)$.
From Proposition \ref{prop6} it follows that the tangent space of $W^r_d(C)$ at $L$ is equal to the tangent space of the scheme defined by the $(g-(d-e'k)+r(0))$-th Fitting ideal of $R^1\pi_*(\mathcal{P}_{d-e'k})$.
In case $Z'$ is of type I then it follows from Corollary \ref{cor1} that the dimension of the tangent space is equal to $\dim (Z')=\dim (Z)$.
In case $Z'$ is of type II then the same conclusion follows from Proposition \ref{prop5}.
\end{proof}

Putting together Corollary \ref{cor1}, Proposition \ref{prop5} and Corollary \ref{cor2} we have proved Theorem B.

\section{A smoothness result}\label{section5}

Up to now, if $Z=\overline{\Sigma}_{\overrightarrow{e}}(C,f)$ is a component of type I of some $W^r_d(C)$, in case $f$ is general then $Z$ and $W^r_d(C)$ are smooth as schemes along $\Sigma_{\overrightarrow {e}}(C,f)$ (this follows from Corollary \ref{cor1} and from \cite{ref7} because of Proposition \ref{prop3}).
The methods of this paper imply we get more strata in the smooth locus of the schemes $Z$ and $W^r_d(C)$.

\begin{proposition}\label{prop7}
Let $\overrightarrow{e'}$ be a splitting type of type $B(0,b,y,u,v)$ with $u,v \neq 0$.
Let $\overrightarrow{e}$ be the splitting type of type $B(0,b',y',v,0)$ obtained by taking $\overrightarrow {e'}-1$ and making it of balanced plus balanced type without changing the positive part.
In case $f$ is general then $\Sigma _{\overrightarrow {e'}-1}(C,f)$ is contained in the smooth locus of the scheme $W^{v-1}_d(C))$ with $d$ determined by Formula \ref{vgl1} using $\overrightarrow {e}$ (in particular it is also contained in the smooth locus of $\overline{\Sigma}_{\overrightarrow {e}}(C,f)$).

\end{proposition}

\begin{proof}
From Proposition \ref{prop4} we know the tangent space to the scheme $\overline{\Sigma}_{\overrightarrow {e'}-1}(C,f)$ at $L\in \Sigma_{\overrightarrow {e'}-1}(C,f)$ is equal to the intersection of the tangent dimensions of the schemes defined by the $(g-(d+k)+r'(1))$-th Fitting ideal of $R^1\pi_*(\mathcal{P}_d(1))$ and the $(g-d+r'(0))$-th Fitting ideal of $R^1\pi_*(\mathcal{P}_d)$ (we use $r'(i)$ related to $\overrightarrow {e'}-1$).
From Proposition \ref{prop3} we know this intersection has dimension equal to $\dim (\overline {\Sigma}_{\overrightarrow {e'}-1}(C,f))$.
It follows $\overline{\Sigma}_{\overrightarrow {e'}-1}(C,f)$ is reduced at $L$ (this is already proved in \cite{ref6}) and locally at $L$ the scheme $\overline{\Sigma}_{\overrightarrow {e'}-1}(C,f)$ is the scheme-theoretic intersection of the two schemes mentioned above.
Indeed, this equation is clear as sets.
Therefore the intersection of those two schemes also has a tangent space at $L$ of dimension equal to the dimension of the intersection of the schemes at $L$.
This implies also the intersection of those two schemes is already reduced at $L$.
We know from \cite{ref6} that the scheme $\overline{\Sigma}_{\overrightarrow {e}}(C,f)$ is reduced and it contains $\Sigma_{\overrightarrow {e'}-1}(C,f)$.
Also because of Definition \ref{definition6} we have $\overline{\Sigma}_{\overrightarrow {e}}(C,f)$ is a closed subscheme of the scheme defined by the $(g-d+r'(0))$-th Fitting ideal of $R^1\pi_*(\mathcal{P}_d)$.
It follows that locally at $L$ the scheme $\overline{\Sigma}_{\overrightarrow {e'}-1}(C,f)$ is the intersection of the scheme $\overline{\Sigma }_{\overrightarrow{e}}(C,f)$ and the scheme defined by the $(g-(d+k)+r'(1))$-th Fitting ideal of $R^1\pi_*(\mathcal{P}_d(1))$.

It follows that locally at $L$ the scheme $\overline{\Sigma}_{\overrightarrow {e'}-1}(C,f)$ is defined as the subscheme of $\overline{\Sigma}_{\overrightarrow {e}}(C,f)$ defined by the $(g-(d+k)+r'(1))$-th Fitting ideal of $R^1\pi_*(\mathcal{P}_d(1))\vert _{\overline{\Sigma}_{\overrightarrow {e}}(C,f)}$.
In \cite{ref1}, Chapter IV, Lemma 3.1, it is proved that this sheaf is equal to $R^1\pi_*(\mathcal{P}_d(1)\vert _{\Sigma _{\overrightarrow {e}}(C,f)})$.
Now we can use the result of Proposition \ref{prop2}, especially Remark \ref{rem1} showing that locally at $[L]$ the scheme $\overline{\Sigma}_{\overrightarrow {e'}-1}(C,f)$ is locally defined by $\codim _{\overline{\Sigma}_{\overrightarrow{e}}(C,f)}\overline{\Sigma}_{\overrightarrow {e'}-1}(C,f)$ equations.
Since the scheme $\overline{\Sigma}_{\overrightarrow {e'}-1}(C,f)$ is  smooth along $\Sigma_{\overrightarrow {e'}-1}(C,f)$ it follows that also the scheme $\overline{\Sigma}_{\overrightarrow {e}}(C,f)$ is smooth at $L$ (its tangent space needs to have dimension $\dim (\overline{\Sigma}_{\overrightarrow {e}}(C,f))$ at $L$).

Remember that $\overline{\Sigma}_{\overrightarrow {e}}(C,f)$ is an irreducible component of $W^{v-1}_d(C)$.
The tangent space of $W^{v-1}_d(C)$ at $L\in \Sigma _{\overrightarrow {e'}-1}(C,f)$ is the tangent space at $L$ of the scheme defined by $(g-d+r'(0))$-th Fitting ideal of $R^1\pi_*(\mathcal{P}_d)$.
Now let $r(i)$ be defined by $\overrightarrow {e}$.
Since $u \neq 0$ we have $r'(0)=r(0)$ and $r'(i)>r(i)$ in case $0 < i < b'$.
Then it follows from \cite{ref9}, Chapter IV, Proposition (4.2) (ii), that the tangent space at $L \in \Sigma_{\overrightarrow {e'}-1}(C,f)$ of the scheme defined by the $(g-(d+ki)+r'(i))$-th Fitting ideal of $R^1 \pi_*(\mathcal{P}_d(i))$ for $0 < i <b'$ is equal to $T_L(\Pic ^d(C))$.
This implies that at such point the tangent space $T_L(\overline{\Sigma}_{\overrightarrow {e}}(C,f)$ is equal to $T_ L(W^{v-1}_d(C))$.
This implies $\dim (T_L(W^{v-1}_d(C)))=\rho ^{v-1}_d(g)$, hence $W^{v-1}_d(C)$ is smooth at $L$.
\end{proof}

\begin{remark}\label{rem7}
The proof of the smoothness of $\overline{\Sigma}_{\overrightarrow {e}}(C,f)$ in Proposition \ref{prop7} is inspired by the proof of Proposition 9.1 in \cite{ref6}.
It would be nice to have a similar argument for the scheme $W^{v-1}_d(C)$ but in order to use Proposition \ref{prop2} we would have to know in advance that the scheme $W^r_d(C)$ is reduced along $\Sigma _{\overrightarrow {e'}-1}(C,f)$.
This is because of the use of Grauert's Theorem in the proof of Proposition \ref{prop2}.
This is a very subtle problem giving rise to interesting questions.

As an example consider the situation with $\overrightarrow{e}$ a splitting type of type $B(0,b,y,1,0)$.
As a set we have $\overline {\Sigma _{\overrightarrow{e}}(C,f)}$ is equal to $W^0_d(C)$.
In case $f$ is general then the scheme $\overline {\Sigma _{\overrightarrow{e}}(C,f)}=\overline {\Sigma} _{\overrightarrow{e}}(C,f)$ is reduced while from \cite{ref9}, Chapter IV, Corollary 4.5, we know also $W^0_d(C)$ as a scheme is reduced, so both schemes are equal.
This scheme $\overline{\Sigma}_{\overrightarrow {e}}(C,f)$ is an irreducible component of type I and for all $1 \leq c \leq b-2$ we have $\overline{\Sigma}_{\overrightarrow {e}+c}(C,f)$ is an irreducible component of $W^c_{d+kc}(C)$.
Again from \cite{ref6} we know the scheme $\overline{\Sigma}_{\overrightarrow {e}+c}(C,f)$ is reduced and therefore $\overline{\Sigma}_{\overrightarrow {e}+c}(C,f) = W^0_d(C)+cM \subset \Pic^{d+ck}(C)$ as schemes.
However we only know $\overline{\Sigma}_{\overrightarrow{e}+c}(C,f)$ is a closed subscheme of $W^c_{d+kc}(C)$.
In case $k$ is small with respect to $g$ and $c$ is sufficiently small then, as a set, $W^c_{d+kc}(C)=\overline{\Sigma}_{\overrightarrow {e}+c}(C,f)$.
Is it true that for general $f$ one has the equality $W^c_{d+kc}(C)=\overline{\Sigma}_{\overrightarrow {e}+c}(C,f)=W^0_d(C)+c[M]$ as schemes in this case?
For $k=2$, the hyperelliptic case, this is proved only quite recently (for any hyperelliiptic curve) in the appendix of \cite{ref15} (see also Proposition 4.14 in \cite{ref16}).
\end{remark}

\textbf {Acknowledgement.} I want to thank some people for some communication concerning (parts of) this paper: Nero Budur, An-Khuong Doan, Gerriet Martens and Hans Schoutens.
In particular I want to thank Hannah Larson for some correspondence we had on the subject of this paper.

\end{document}